\documentclass[a4paper,reqno, 10pt]{amsart}

\usepackage{amsmath,amssymb,amsfonts,amsthm, mathrsfs}
\usepackage[latin1]{inputenc}
\usepackage{lmodern}

\usepackage{verbatim,wasysym,cite}

\usepackage[latin1]{inputenc}
\usepackage{microtype}
\usepackage{color,enumitem,graphicx}
\usepackage[colorlinks=true,urlcolor=blue, citecolor=red,linkcolor=blue,
linktocpage,pdfpagelabels, bookmarksnumbered,bookmarksopen]{hyperref}
\usepackage[english]{babel}
\usepackage[symbol]{footmisc}
\renewcommand{\epsilon}{{\varepsilon}}

\numberwithin{equation}{section}
\newtheorem{theorem}{Theorem}[section]
\newtheorem{lemma}[theorem]{Lemma}
\newtheorem{remark}[theorem]{Remark}

\newtheorem{proposition}[theorem]{Proposition}

\newtheorem{claim}[theorem]{Claim}

\newcommand{\C}{\mathbb C}

\newcommand{\R}{\mathbb R}
\newcommand{\N}{\mathbb N}

\def\({\left(}
\def\){\right)}
\def\<{\left\langle}
\def\>{\right\rangle}

\def\Sch{{\mathcal S}}

\def\eps{\varepsilon}

\DeclareMathOperator{\RE}{Re}
\DeclareMathOperator{\IM}{Im}

\newcommand{\qtq}[1]{\quad\text{#1}\quad}

\begin{document}

\title[NLS with a repulsive inverse power potential]
{Mass-energy threshold dynamics for the focusing  NLS with a repulsive inverse-power potential}

\author[Alex H. Ardila]{Alex H. Ardila}
\address{Department of Mathematics, Universidade Federal de Minas Gerais\\ ICEx-UFMG\\ CEP
30123-970\\ MG, Brazil} 
\email{ardila@impa.br}

 \author[Masaru Hamano]{Masaru Hamano}
\address{Department of Mathematics, Graduate School of Science and Engineering Saitama University,\\
Shimo-Okubo 255, Sakura-ku, Saitama-shi, Saitama 338-8570, Japan} 
\email{m.hamano.733@ms.saitama-u.ac.jp/ess70116@mail.saitama-u.ac.jp}

\author[Masahiro Ikeda]{Masahiro Ikeda}
\address{Department of Mathematics, Faculty of Science and Technology, Keio University, 3-14-1 Hiyoshi,
Kohoku-ku, Yokohama, 223-8522, Japan/Center for Advanced Intelligence Project, Riken, Japan} 
\email{masahiro.ikeda@keio.jp/masahiro.ikeda@riken.jp}

\begin{abstract}
In this paper we study long time dynamics (i.e., scattering and blow-up) of solutions for the 
focusing  NLS  with a repulsive inverse-power potential and with initial data 
lying exactly at the mass-energy threshold, namely, when
 $E_{V}(u_{0})M(u_{0})=E_{0}(Q)M(Q)$. Moreover, we  prove failure of the uniform space-time bounds 
at the mass-energy threshold.
\end{abstract}


\maketitle

\medskip

\section{Introduction}
\label{sec:intro}
In this paper we consider the long time dynamics for the following nonlinear Schr\"odinger equation with a repulsive
inverse-power potential
\begin{equation}\label{NLS}\tag{NLS$_{a}$}
\begin{cases}
i\partial_{t}u+\Delta u-a{|x|^{-\mu}}u+|u|^{2}u=0,\\
u(0,x)=u_{0}\in H^{1}(\R^{3}),
\end{cases}
\end{equation}
where  $u=u(t,x)$ is a complex-valued function of $(t,x)\in \mathbb{R}\times\mathbb{R}^{3}$, $1<\mu<2$ and $a>0$.
We define the operator $H=-\Delta +V(x)$, where 
\[\text{$V(x)=a{|x|^{-\mu}}$ with $a>0$ and $1<\mu<2$.}\]
We define the energy functional on $H^{1}(\R^{3})$ as follows:
\[
E_{V}(u)=\int_{\R^{3}}\tfrac{1}{2}|\nabla u|^{2}+\tfrac{1}{2}V(x)|u|^{2}
-\tfrac{1}{4}|u|^{4}dx.
\]
Note that $E_{V}$ is the generation Hamiltonian of \eqref{NLS}. The Cauchy problem for the 
present equation has been studied by Guo, Wang and Yao\cite{GuoWangYao2018} (see also \cite{2HamanoIkeda2022}), more precisely:
for $u_{0}\in H^{1}({\mathbb{R}}^{3})$, there exist $T_{\ast}=T(\|u_{0}\|_{H^{1}})>0$ and a unique solution  $u\in C([0, T_{\ast}), H^{1}({\mathbb{R}^{3}}))$ of the Cauchy problem \eqref{NLS}. Furthermore,  the solution satisfies the conservation of energy and mass
\begin{equation*}
E_{V}(u(t))=E_{V}(u_{0})\quad  \text{and} \quad M(u(t))=M(u_{0}),
\end{equation*}
for all $t\in [0,T_{\ast})$, where
\[
M(u)=\tfrac{1}{2}\int_{\R^{3}}|u|^{2}dx.
\]
Scattering and blow-up for large date were studied for the NLS with with a repulsive
inverse-power potential in several papers in different contexts; see 
\cite{HamanoIkeda20, 2HamanoIkeda2022, GuoWangYao2018, KillipMurphyVisanZheng, LU20183174, ZhangZheng2014, Dinh2021N, MiaoZhanZheng2018} and references therein. In particular, the ground state 
solution of the free cubic nonlinear Schr\"odinger equation (i.e., \eqref{NLS} with $a=0$) plays an important role in the behavior (scattering/blow-up) of solutions for \eqref{NLS}. Recall that the ground state is the unique, radial, vanishing at infinity and positive solution 
of the following nonlinear elliptic equation
\begin{equation}\label{ELS}
-\Delta Q+Q-Q^{3}=0.
\end{equation}
We defined the Sobolev space adapted to $H$ by
\[
\|u\|_{\dot{H}^{1}_{V}}^{2}=\< Hu, u \>=\int_{\R^{3}}|\nabla u|^{2}+V(x)|u|^{2}dx.
\]
In \cite{2HamanoIkeda2022}, the authors have studied the global existence and scattering of solutions to \eqref{NLS} when the initial
data has nonnegative virial functional $P_{V}$,  where 
\begin{align*}
	P_{V}(u)&=2\| \nabla u\|^{2}_{L^{2}}-\int_{\R^{3}}(x\cdot \nabla V)|u(x)|^{2}dx-\tfrac{3}{2} \|u\|^{4}_{L^{4}}\\
          &=2\| \nabla u\|^{2}_{L^{2}}+\mu \int_{\R^{3}}\tfrac{a}{|x|^{\mu}}|u(x)|^{2}dx-\tfrac{3}{2} \|u\|^{4}_{L^{4}}.
\end{align*}
More specifically, we have the following result.
\begin{theorem}[Sub-threshold scattering, \cite{HamanoIkeda20, GuoWangYao2018}]\label{Th1}
Fix $1<\mu<2$ and $a>0$. Let $u(t)$ be the corresponding solution to \eqref{NLS} with initial data $u_{0}\in H^{1}(\R^{3})$. If $u_{0}$ obeys
\begin{equation}\label{SubT}
E_{V}(u_{0})M(u_{0})<E_{0}(Q)M(Q)
\quad\text{and}\quad
P_{V}(u_{0})\geq 0,
\end{equation}
 then the solution $u(t)$ exists globally and scatters in $H^{1}(\R^{3})$.
\end{theorem}
In the case $\mu=2$, Killip, Murphy, Visan and J. Zheng \cite{KillipMurphyVisanZheng}  proved a similar scattering result for $a>-\frac{1}{4}$.

The theorem above is a consequence of the fact that the solutions to \eqref{NLS} obeys the global spacetime bound
\begin{equation}\label{GlobalBoundB}
\|u\|_{L^{5}_{t,x}(\R\times\R^{3})}<C(E_{V}(u_{0}), M(u_{0}), E_{0}(Q), M(Q)),
\end{equation}
for some $C:(0,  E_{0}(Q)M(Q))\rightarrow (0, \infty)$.

 In our first result we show that  Theorem~\ref{Th1} is sharp, i.e., the constant $C(\cdot)$ diverges as 
we approach the mass-energy threshold. Indeed,

\begin{theorem}[Failure of uniform space-time bounds at threshold.]\label{BoundTheorem} 
Fix $1<\mu<2$ and $a>0$.
There exists a sequence of global solutions $u_{n}$ of \eqref{NLS} such that
\[
E_{V}(u_{n})M(u_{n})\nearrow  E_{0}(Q)M(Q)
\quad \text{and}\quad 
P_{V}(u_{n}(0))\rightarrow 0,
\]
as $n\to \infty$ with
\[
\lim_{n\to \infty}\|u_{n}\|_{L^{5}_{t,x}(\R\times\R^{3})}=\infty.
\]
\end{theorem}

The purpose of this paper is to study the long time dynamics (i.e., scattering and blow-up) for \eqref{NLS}
exactly at the mass-energy threshold, i.e., when $E_{V}(u_{0})M(u_{0})=E_{0}(Q)M(Q)$.
We now state the main result of this paper.

\begin{theorem}[Threshold dynamics]\label{Th2} Fix $1<\mu<2$ and $a>0$.
Let $u(t)$ be the corresponding solution to \eqref{NLS} with initial data $u_{0}\in H^{1}(\R^{3})$. 
\begin{enumerate}[label=\rm{(\roman*)}]
\item If $u_{0}$ obeys
\begin{equation}\label{Thres}
E_{V}(u_{0})M(u_{0})=E_{0}(Q)M(Q) \quad \text{and} \quad P_{V}(u_{0})\geq 0,
\end{equation}
then the solution $u(t)$ to \eqref{NLS} is global and $u\in L^{5}_{t, x}(\R\times\R^{3})$.
Consequently, the solution $u$ scatters in both directions.
\item If $u_{0}$ obeys
\begin{equation}\label{BlowC}
E_{V}(u_{0})M(u_{0})=E_{0}(Q)M(Q) \quad \text{and} \quad P_{V}(u_{0})< 0,
\end{equation}
and $xu_{0}\in L^{2}(\R^{3})$ or $u_{0}$ is radially symmetric,
then  the solution $u$ to \eqref{NLS} blows up in both time directions.
\end{enumerate}
\end{theorem}

In the case $a=0$, a similar result was originally proven by Duyckaerts-Roudenko \cite{DuyckaertsRou2010}.
However, due to the presence of the potential, the method developed in \cite{DuyckaertsRou2010} cannot be applied to \eqref{NLS}. 
To overcome this problem, the proof of scattering result in Theorem~\ref{Th2} is based on the work of  Miao, Murphy 
and Zheng\cite{MiaMurphyZheng2021}. An analogous result to Theorem~\ref{Th2}~(i) for the NLS in the exterior of a convex obstacle 
was obtained by \cite{DuyLanRou2022}.  Recently, the same argument have been applied to the 
focusing NLS with a repulsive Dirac delta potential; see \cite{ArdilaInui2022} for more details.
On the other hand, our proof of blow-up result is based on the argument developed in \cite{DuyckaertsRou2010, TInu2022, Ardila2022}.
For more details, we refer to Section~\ref{Sec22}.

\begin{remark}
Recently, in \cite[Theorem 7.2.]{Dinh2021N}, using the argument of Dodson and Murphy \cite{DodsonMurphy2018} the author shows that
under condition \eqref{Thres}  the corresponding solution $u(t)$ to Cauchy problem \eqref{NLS} either (i) scatters in $H^{1}(\R^{3})$ 
forward in time, or (ii) there exist $\left\{t_{n}\right\}_{n\in\N}\subset \R$ with $t_{n}\to \infty$, 
$\left\{y_{n}\right\}_{n\in\N}\subset \R^{3}$ with  $|y_{n}|\to \infty$ and $\lambda, \theta\in \R$ so that
$u(t_{n}, \cdot+y_{n})\to e^{i \theta}\lambda Q$ strongly in $H^{1}(\R^{3})$.
As a consequence of the Theorem~\ref{Th2}~(i), we can rule out the second possibility.
\end{remark}

This present paper is organized as follows. In Section~\ref{S:preli} we give some results that are necessary for later sections.
In particular, the linear profile decomposition, the stability result to \eqref{NLS}, localized Virial identities, and variational analysis
of the ground state related to \eqref{ELS}. In Section~\ref{S:Compactness} we show that if the scattering 
result of Theorem~\ref{Th2} fails, then  we can find a forward global 
solution $u\in C([0, \infty); H^{1}(\R^{3}))$ to \eqref{NLS}  which satisfies that $\left\{u(t, \cdot+x_{0}(t)):t\in [0, \infty)\right\}$ is pre-compact in $H^{1}(\R^{3})$ for some   function $x_{0}:[0, \infty) \to \R$ (cf. Proposition~\ref{Criticalsolution}).
In Section~\ref{S:Modulation} we discuss modulation (Proposition~\ref{Modilation11}). In Section~\ref{S:Impossibility}, using the result of Section~\ref{S:Compactness} (Proposition~\ref{Criticalsolution}) and adopting the method of Miao, Murphy and Zheng\cite{MiaMurphyZheng2021} we establish  the scattering part of Theorem~\ref{Th2}.
Section~\ref{Sec22} is devoted to the proof of the blow-up result given in Theorem~\ref{Th2}.
Finally, in Section~\ref{S:Fail} we prove Theorem~\ref{BoundTheorem}.

\textbf{Notations.} Given two positive quantities $A$, $B$ we write $A\lesssim B$ or $B\gtrsim A$  to signify $A\leq CB$ for some postive constant $C>0$. When  $A \lesssim B \lesssim A$,  we write $A\sim B$. Recall that $H:=-\Delta +V(x)$, where $V(x)=a|x|^{-\mu}$ with $a>0$.
We write
\[
\|u\|_{\dot{H}^{1}_{V}}^{2}:=\< Hu, u \>=\int_{\R^{3}}|\nabla u|^{2}+V(x)|u|^{2}dx,
\]
and $\|u\|_{{H}^{1}_{V}}^{2}=\|u\|_{\dot{H}^{1}_{V}}^{2}+\|u\|_{L^{2}}^{2}$.

Throughout the paper, we will use the spaces $\dot{S}^{s}(I)$ for $s\geq 0$,
\[
\dot{S}_{V}^{s}(I)=L^{\infty}_{t} \dot{H}^{s}_{V}(I\times \R^{3})\cap L^{2}_{t} \dot{H}^{s, 6}_{V}(I\times \R^{3}),
\quad
{S}_{V}^{s}(I)=L^{\infty}_{t} {H}^{s}_{V}(I\times \R^{3})\cap L^{2}_{t} {H}^{s, 6}_{V}(I\times \R^{3}).
\]
Finally, for $f\in H^{1}(\R^{3})$ we denote
\[
\delta(f):=\|Q\|_{\dot{H}^{1}}^{2}-\|f\|_{\dot{H}^{1}_{V}}^{2}.
\]

\section{Preliminaries}\label{S:preli}
In this section  we review the tools that will be needed in the proof of Theorems~\ref{BoundTheorem} and \ref{Th2}.
\subsection{Cauchy problem and profile decomposition}

First, we have the following result.

\begin{proposition}[Theorem 1.1 in \cite{GuoWangYao2018}]\label{Condition-scatte}
Fix $a>0$ and $0<\mu< 2$. Let  $u_{0}\in H^{1}(\R^{3})$ and $u(t)$ be the corresponding solution of Cauchy problem \eqref{NLS}.
If $u$ is a global  solution to \eqref{NLS} with $\|u\|_{L^{5}_{t,x}(\R \times \R^{3})}< \infty$,
then $u(t)$ scatters in  $H^{1}$.
\end{proposition}

\begin{remark}[Existence of wave operators; Theorem 1.1 in \cite{GuoWangYao2018}]\label{Exitencescat}
From Theorem 1.1~(iii) in \cite{GuoWangYao2018}, we have that given $\psi\in H^{1}(\R^{3})$,
there exist $T>0$ and a solution $v\in C([T,\infty), H^{1}(\R^{3}))$ to \eqref{NLS} such that
\[
\|v(t)-e^{itH}\psi\|_{H^{1}}\rightarrow 0 \quad\text{ as $t\to\infty$}.
\]
A similar result holds in the negative time direction.
\end{remark}

\begin{proposition}[Linear profile decomposition; Lemma 2.12 in \cite{GuoWangYao2018}] 
\label{LPD}
Fix $a>0$ and $1<\mu < 2$. Let $\{ \varphi_n\}_{n\in \N}$ be a bounded sequence in $H^1(\mathbb{R}^{3})$. Then, up to subsequence, we have the decomposition 
\begin{align*}
	\varphi_n=\sum_{j=1}^{J} e^{i t_{n}^j H} \tau_{x_n^j} \psi^j +R_n^J, \quad \forall J \in \N, 
\end{align*}
where 
$t_n^j\in \mathbb{R}$, $x_n^j \in \mathbb{R}^{3}$, $\psi^j \in H^1(\mathbb{R}^{3})\setminus\{0\}$, and the following statements hold. 
\begin{itemize}
\item for any fixed $1\leq j\leq J$, 
\begin{align*}
&\text{either } t_n^j=0 \text{ for any } n \in \N, \text{ or } t_n^j \to \pm \infty \text{ as } n\to \infty;
\\
&\text{either } x_n^j=0 \text{ for any } n \in \N, \text{ or } |x_n^j| \to +\infty \text{ as } n\to \infty.
\end{align*}
\item orthogonality of the parameters: namely, for $1\leq j\neq k\leq J$
\[ |t_n^j -t_n^k|+|x_n^j-x_n^k| \to \infty \text{ as } n \to \infty. \]
\item asymptotic smallnes property
\[ \forall \eps >0, \exists J=J(\eps) \in \N \text{ such that } \limsup_{n\to \infty} \| e^{-itH}R_n^J\|_{L_{t,x}^5} <\eps. \]
\item asymptotic Pythagorean expansions: for any $ J\in \N$
\begin{align*} 
\| \varphi_n\|_{L^2}^2 &=\sum_{j=1}^J \|\psi^j\|_{L^2}^2 +\| R_n^J\|_{L^2}^2 +o_n(1),
\\
\| \varphi_n\|_{\dot{H}^{1}_{V}}^2 &=\sum_{j=1}^J \| \tau_{x_n^j} \psi^j \|_{\dot{H}^{1}_{V}}^2 +\| R_n^J\|_{\dot{H}^{1}_{V}}^2 +o_n(1).
\end{align*}
Moreover, we have 
\[ \| \varphi_n\|_{L^4}^4 =\sum_{j=1}^J \| e^{i t_n^j H} \tau_{x_n^j} \psi^j \|_{L^4}^4 +\| R_n^J\|_{L^4}^4 +o_n(1)   \quad \forall J\in \N. \]
\end{itemize}
\end{proposition}

For the following result, recall that for $s\geq 0$,
\[
\dot{S}_{V}^{s}(I)=L^{\infty}_{t} \dot{H}^{s}_{V}(I\times \R^{3})\cap L^{2}_{t} \dot{H}^{s, 6}_{V}(I\times \R^{3}),
\quad
{S}_{V}^{s}(I)=L^{\infty}_{t} {H}^{s}_{V}(I\times \R^{3})\cap L^{2}_{t} {H}^{s, 6}_{V}(I\times \R^{3}).
\]

\begin{lemma}[Stability; Lemma 2.3 in \cite{GuoWangYao2018} and Theorem 4.10 in \cite{2HamanoIkeda2022}]\label{StabilityNLS}
Fix $a>0$ and $0<\mu<2$. Let $I\subset \R$ be a time interval containing $t_{0}$ and  let $\tilde{u}$ satisfy
\[ 
(i\partial_{t}-H) \tilde{u}=-|\tilde{u}|^{2}\tilde{u}+e, \quad 
\tilde{u}(t_{0})=\tilde{u}_{0}
\]
on $I\times\R^3$ for some function $e:I\times\R^{3}\rightarrow \C$. Fix $u_{0}\in H^{1}(\R^{3})$ and suppose
\[
\| \tilde{u}_{0}  \|_{H^{1}}+\| {u}_{0}  \|_{H^{1}}\leq E\qtq{and} 
	\| \tilde{u}  \|_{L_{t,x}^{5}(I\times\R^{3})}\leq L
\]
for some $E$, $L>0$.  
Assume  the smallness conditions 
\[
	\|u_{0}-\tilde{u}_{0} \|_{\dot{H}^{\frac{1}{2}}}\leq \epsilon\qtq{and}	\||\nabla|^{\frac{1}{2}} e  \|_{N(I)}\leq \epsilon,
\]
for some $0<\epsilon<\epsilon_{0}=\epsilon_{0}(\mbox{E,L})>0$. Here,
\[
N(I)=L^{1}_{t}L^{2}_{x}(I\times \R^{3})+L^{\frac{10}{7}}_{t,x}(I\times \R^{3})+L^{\frac{5}{3}}_{t}L^{\frac{30}{23}}_{x}(I\times \R^{3}).
\]
Then there exists a unique solution $u$ to  \eqref{NLS} with initial data $u_{0}$ at the time $t=t_{0}$ satisfying
\[
	\|u-\tilde{u}\|_{\dot{S}^{\frac{1}{2}}(I\times \R^{3})}\leq C(E,L)\epsilon
	\qtq{and}	\|u\|_{{S}^{1}(I\times \R^{3})}\leq C(E,L).
\]
\end{lemma}

For $a\geq 0$, we define on $H^{1}(\R^{3})$ the following functional (recall that $V(x)=a{|x|^{-\mu}}$):
\begin{align}\label{FS}
S_{V}(f)&=E_{V}(f)+\tfrac{1}{2}\|f\|^{2}_{L^2}
=\tfrac{1}{2}\| f\|^{2}_{\dot{H}^{1}_{V}}+\tfrac{1}{2}\|f\|^{2}_{L^2}-\tfrac{1}{4}\|f\|_{L^{4}}^{4}, \qtq{for $f\in H^{1}$.}
\end{align}

\begin{lemma}[Embedding nonlinear profiles; Lemma 2.13 in \cite{GuoWangYao2018}]\label{P:embedding}
Fix $a>0$ and $1<\mu<2$.
Let $\{t_n\}_{n\in \N}$ satisfy $t_n\equiv 0$ or $t_n\to\pm\infty$, and let $\{x_n\}_{n\in \N}\subset \R^{3}$ satisfy $|x_{n}| \to \infty$.
Suppose $\phi\in H^{1}(\R^{3})$ obeys 
\begin{align}\label{PriC}
	&S_{0}(\phi)<S_{0}(Q) \qtq{and} P_{0}(\phi)\geq 0
	\qtq{if $t_{n}\equiv0$}
\end{align}
or
\begin{align}\label{seC}
	\tfrac{1}{2}\|\phi\|_{{H}^{1}}<S_{0}(Q) \qtq{if $t_{n}\to \pm \infty$.}
\end{align}
Then for $n$ sufficiently large, there exists a global solution $v_n$ to \eqref{NLS} so that
\[
v_n(0)=\phi_n \qtq{and} \|v_n\|_{{S}_{V}^{s}(\R)}\lesssim_{\|\phi\|_{H^{1}}} 1,
\]
where
\[
\phi_n(x)=e^{-it_{n}H}\tau_{x_{n}}\phi(x).
\]
Furthermore, for any $\epsilon>0$ there exist $N=N(\epsilon)\in \N$ and a smooth compactly supported function
$\chi_{\epsilon}\in C^{\infty}_{c}(\R\times \R^{3})$  such that for  $n\geq N$, we have
\begin{align}\label{BewAprox11}
\big\| v_{n}(t,x)-\chi_{\epsilon}(t+t_{n}, x-x_{n})  \big\|_{X(\R\times \R^{3})}
&<\epsilon,
\end{align}
 where 
\[
X\in\{L_{t,x}^{5}, L_{t,x}^{\frac{10}{3}},  L_{t}^{\frac{30}{7}}L^{\frac{90}{31}}_{x},   
L^{\frac{30}{7}}_{t}\dot{H}_{x}^{\frac{31}{60},\frac{90}{31}} \}.
\]
\end{lemma}

\begin{lemma}[Hardy's inequality, \cite{ZhangZheng2014}]\label{GHI}
Fix $1<p<\infty$ and $0<\mu<3$. Then, the
following inequality holds
\[
\int_{\R^{3}}\tfrac{|u(x)|^{p}}{|x|^{\mu}}dx\lesssim_{p,\mu}
\||\nabla|^{\frac{\mu}{p}} u\|^{p}_{L^{p}}.
\]
In particular, if $0<\mu<2$, then we have that the embedding $H^{1}\hookrightarrow L^{2}(\sqrt{V}dx)$ is continuous.
\end{lemma}

\subsection{Varational analysis}
First, we recall here some well-known properties of the ground state. We have the following sharp Gagliardo-Nirenberg inequality, 
\begin{equation}\label{GI}
\|f\|^{4}_{L^{4}}\leq C_{GN}\|\nabla f\|^{3}_{L^{2}}\|f\|_{L^{2}},
\end{equation}
where 
\begin{equation}
\label{C_GN}
C_{GN}=\tfrac{\|Q\|^{4}_{L^{4}}}{\|\nabla Q\|^{3}_{L^{2}}\|Q\|_{L^{2}}}
\end{equation}

It is well-known that the ground state $Q$ satisfies the Pohozaev's identities
\begin{equation}\label{PoQ}
E_{0}(Q)=\tfrac{1}{2}\| Q\|^{2}_{L^{2}}=\tfrac{1}{6}\|\nabla Q\|^{2}_{L^{2}}=\tfrac{1}{8}\| Q\|^{4}_{L^{4}}.
\end{equation}
Moreover, by straightforward calculations we deduce
\begin{align}\label{IDbG}
	[C_{GN} \|Q\|_{L^{2}}]^{-\frac{2}{3}}
=\tfrac{3}{4} \|Q\|^{\frac{4}{3}}_{L^{4}}.
\end{align}

For $a\geq0$, we define the following variational problem (recall that $V(x)=a{|x|^{-\mu}}$):

\begin{equation}\label{dinf}
d_{a}:=\inf\left\{S_{V}(\varphi): \varphi\in H^{1}(\R^{3})\setminus\left\{{0}\right\}, P_{V}(\varphi)=0 \right\},
\end{equation}
 where the functional $S_{V}$ is given by \eqref{FS}.

For the proof of the following lemma see \cite[Lemmas 3.5 and 3.6]{GuoWangYao2018}

\begin{lemma}\label{minimumS}
Fix $a>0$ and $0<\mu<2$. Then $d_{a}$ is never attained for any $a>0$. Moreover, $d_{a}=S_{0}(Q)$.
\end{lemma}

\begin{lemma}\label{GlobalS}
Fix $a>0$ and $0<\mu<2$.
Assume that $u_{0}\in H^{1}(\R^{3})$ satisfies
\begin{equation}\label{Condition11}
E_{V}(u_{0})=E_{0}(Q), \quad M(u_{0})=M(Q)
\quad \text{and}\quad 
P_{V}(u_{0})\geq 0.
\end{equation}
Then the corresponding solution $u(t)$ to \eqref{NLS} is global and satisfies 
\begin{equation}\label{PositiveP}
P_{V}(u(t))> 0 \quad \text{for all $t\in \R$}.
\end{equation}
Moreover,  we have 
\begin{align}\label{EquiNorm}
&\sup_{t\in \R}\|u(t)\|^{2}_{H^{1}_{V}}
\sim
 S_{0}(Q)\\ \label{GINequ}
&\|u(t)\|^{2}_{\dot{H}^{1}_{V}}<\|Q\|^{2}_{\dot{H}^{1}}
\quad \text{for all $t\in \R$.}
\end{align}
\end{lemma}
\begin{proof}
Notice that by  \eqref{Condition11} we get $S_{V}(u(t))=S_{0}(Q)$, where the functional $S_{V}$ is given by \eqref{FS}.
By contradiction, suppose that there exists $t_{0}\in \R$ so that $P_{V}(u(t_{0}))=0$. Then $u(t_{0})$ is a minimizer of $d_{a}$ (cf. \eqref{dinf}),  which is a contradiction with Lemma \ref{minimumS}. Thus, 
\begin{equation}\label{PPositive}
\text{$P_{V}(u(t))> 0$ for all $t$ in the existence time.}
\end{equation}

Next, notice that $2S_{0}(Q)\leq \|u(t)\|^{2}_{{H}^{1}_{V}}$ for all $t$ in the existence time. On the other hand, by using \eqref{PPositive} we infer that (recall that $0<\mu<2$)
\[
\|u(t)\|^{2}_{\dot{H}^{1}_{V}}+3\|u(t)\|^{2}_{L^{2}}
< \|u(t)\|^{2}_{\dot{H}^{1}_{V}}+3\|u(t)\|^{2}_{L^{2}}+P_{V}(u(t))\leq 6S_{V}(u(t))=6S_{0}(Q)
\]
for all $t$ in the existence time, which implies that $u$ is global and satisfies \eqref{EquiNorm}. 
Finally, as $6S_{0}(Q)=\|Q\|^{2}_{\dot{H}^{1}}+3\|Q\|^{2}_{L^{2}}$ and $M(Q)=M(u_{0})$, by inequality 
above we obtain \eqref{GINequ}.

This completes the proof of lemma.
\end{proof}

\begin{lemma}\label{CondiBlo}
Fix $a>0$ and $0<\mu<2$.
Assume that $u_{0}\in H^{1}(\R^{3})$ satisfies
\begin{equation}\label{Condition22}
E_{V}(u_{0})=E_{0}(Q), \quad M(u_{0})=M(Q)
\quad \text{and}\quad 
P_{V}(u_{0})< 0.
\end{equation}
Then the corresponding solution $u(t)$ to \eqref{NLS} satisfies 
$P_{V}(u(t))<0$  for all $t$ in the existence time. Furthermore,
\begin{align}\label{BelG}
	\|u(t)\|^{2}_{\dot{H}^{1}_{V}}>\|Q\|^{2}_{\dot{H}^{1}}\qtq{for all $t$ in the existence time.}
\end{align}
\end{lemma}
\begin{proof}
Following the same argument as Lemma~\ref{GlobalS}, we can prove $P_{V}(u(t))<0$  for all $t$ in the existence time.

Next, suppose that $\|u(t_{0})\|^{2}_{\dot{H}^{1}_{V}}\leq \|Q\|^{2}_{\dot{H}^{1}}$ for some $t_{0}$. Then, as $M(u(t_{0}))=M(Q)$,
by the Gagliardo-Nirenberg inequality and \eqref{GI}-\eqref{PoQ} we obtain
\begin{align*}
	\tfrac{1}{2}P_{V}(u(t_{0}))&\geq 
	\| \nabla u(t_{0})\|^{2}_{L^{2}}+
	\tfrac{\mu}{2} \int_{\R^{3}}\tfrac{a}{|x|^{\mu}}|u(x, t_{0})|^{2}dx-
	\tfrac{3 C_{GN}}{4} \| \nabla u(t_{0})\|^{3}_{L^{2}}\| u(t_{0})\|_{L^{2}}\\
	&=\| \nabla u(t_{0})\|^{2}_{L^{2}}\(1-\tfrac{3 C_{GN}}{4} \| \nabla u(t_{0})\|_{L^{2}}\| u(t_{0})\|_{L^{2}}\)
	+\tfrac{\mu}{2} \int_{\R^{3}}\tfrac{a}{|x|^{\mu}}|u(x, t_{0})|^{2}dx\\
	&\geq\| \nabla u(t_{0})\|^{2}_{L^{2}}\(1-\tfrac{3 C_{GN}}{4} \| \nabla Q\|_{L^{2}}\| Q\|_{L^{2}}\)+
	\tfrac{\mu}{2} \int_{\R^{3}}\tfrac{a}{|x|^{\mu}}|u(x, t_{0})|^{2}dx\\
	&=\tfrac{\mu}{2} \int_{\R^{3}}\tfrac{a}{|x|^{\mu}}|u(x, t_{0})|^{2}dx>0,
\end{align*}
which is a contradiction. This proves the lemma. 
\end{proof}

\subsection{Virial identities}
\begin{equation}\label{WRR}
w_{R}(x)=R^{2}\phi\(\tfrac{x}{R}\)
\quad \text{and}\quad w_{\infty}(x)=|x|^{2},
\end{equation}
where $\phi$ is a real-valued and radial function so that
\[
\phi(x)=
\begin{cases}
|x|^{2},& \quad |x|\leq 1\\
0,& \quad |x|\geq 2,
\end{cases}
\quad \text{with}\quad 
|\partial^{\alpha}\phi(x)|\lesssim |x|^{2-|\alpha|}.
\]
We introduce the localized virial functional
\[
I_{R}[u]=2\IM\int_{\R^{3}} \nabla w_{R}(x) \cdot\nabla u(t,x) \overline{u(t,x)} dx.
\]

We need the following lemma; see e.g., \cite{GuoWangYao2018}.

\begin{lemma}\label{VirialIden}
Let $R\in [1, \infty]$. Suppose $u(t)$ solves \eqref{NLS}. Then we have
\begin{equation}\label{LocalVirial}
\frac{d}{d t}I_{R}[u]=F_{R, V}[u(t)],
\end{equation}
where
\begin{align*}
F_{R, V}[u]&:=\int_{\R^{3}}(- \Delta \Delta w_{R})|u|^{2}-
\Delta[w_{R}(x)]|u|^{4}+4\RE \overline{u_{j}} u_{k} \partial_{jk}[w_{R}]dx\\
&-2\int_{\R^{3}}|u|^{2}\nabla w_{R}\cdot \nabla V dx\\
&=F_{R,0}[u]-2\int_{\R^{3}}|u|^{2}\nabla w_{R}\cdot \nabla V dx.
\end{align*}
In particular, when $R=\infty$ we have $F_{\infty, V}[u]=4P_{V}(u)$.
\end{lemma}

The proofs of the next two lemmas are very similar to the ones in \cite[Lemmas 2.9 and 2.10]{MiaMurphyZheng2021}.

\begin{lemma}[Lemmas 2.9 in \cite{MiaMurphyZheng2021}]\label{Virialzero}
Consider $R\in [1, \infty]$, $\theta\in \R$ and $y\in\R$. Then we have
\[
F_{R,0}[e^{i\theta}Q(\cdot-y)]=0.
\]
\end{lemma}

\begin{lemma}[Lemmas 2.10 in \cite{MiaMurphyZheng2021}]\label{VirialModulate}
Let $R\in [1, \infty]$, $\chi: I\to \R$, $\theta: I\to \R$,
$y: I\to \R$. Then if $u$ is a solution  to \eqref{NLS} on an interval $I$ we have
\begin{align}\nonumber
	\frac{d}{d t}I_{R}[u]&=F_{\infty,0}[u(t)]\\ \label{Modu11}
	                     &+F_{R, V}[u(t)]-F_{\infty,0}[u(t)]\\\label{Modu22}
											 &-\chi(t)\big\{F_{R,0}[e^{i\theta(t)}Q(\cdot-y(t))]-F_{\infty,0}[e^{i\theta(t)}Q(\cdot-y(t))]\big\},
\end{align}
for all $t\in\R$.
\end{lemma}

We need the following Cauchy-Schwarz inequality; a similar inequality is obtained in \cite[Claim 5.4]{DuyckaertsRou2010}; see also \cite[Lemma 2.4]{TInu2022} and \cite[Lemma 2.2]{Ardila2022}.
\begin{lemma}\label{OtherGN}Fix $a>0$.
Let $f\in H^{1}(\R^{3})$ such that $|x|f\in L^{2}(\R^{3})$. If
\begin{equation}\label{OAss1}
M(f)=M(Q)\quad \text{and}\quad  E_{V}(f)=E_{0}(Q),
\end{equation}
then
\[
\(\IM \int_{\R^{3}} (x\cdot \nabla f) \overline{f} dx\)^{2}
\lesssim
|\delta(f)|^{2}
\int_{\R^{3}}|x|^{2}|f|^{2}dx.
\]
\end{lemma}
\begin{proof}
Given $f\in H^{1}(\R^{3})$ and $\lambda\in \R$ we see that (cf. \eqref{GI}) 
\[
\|f\|_{L^{4}}^{4}\leq C_{GN}\|e^{i\lambda |x|^{2}}f\|_{\dot{H}^{1}_{V}}^{3}\| f\|_{L^{2}}.
\]
As
\[
\|e^{i\lambda |x|^{2}}f\|^{2}_{\dot{H}^{1}_{V}}
=4\lambda^{2}\int_{\R^{3}}|x|^{2}|f|^{2}dx+4\lambda\IM\int_{\R^{3}} (x\cdot \nabla f) \overline{f}dx+
\int_{\R^{3}}|\nabla f|^{2}dx+\int_{\R^{3}}V(x)|f|^{2}dx
\]
we get
\[
\begin{split}
4\lambda^{2}\int_{\R^{3}}|x|^{2}|f|^{2}dx+4\lambda\IM\int_{\R^{3}} (x\cdot \nabla f) \overline{f}dx+
\int_{\R^{3}}|\nabla f|^{2}dx\\
+\int_{\R^{3}}V(x)|f|^{2}dx
- \( \tfrac{\|f\|_{L^{4}}^{4}}{C_{GN}\| f\|_{L^{2}}}\)^{\frac{2}{3}}
\geq 0.
\end{split}
\]
Since the left-hand side of inequality above is a quadratic polynomial in $\lambda$, it follows that the discriminant of this polynomial is 
non-positive, which implies
\begin{equation}\label{BoundI}
\(\IM \int_{\R^{3}} (x\cdot \nabla f) \overline{f} dx\)^{2}\leq
\int_{\R^{3}}|x|^{2}|f|^{2}dx
\(\|f\|_{\dot{H}^{1}_{V}}^{2}-\( \tfrac{\|f\|_{L^{4}}^{4}}{C_{GN}\| f\|_{L^{2}}}\)^{\frac{2}{3}}\).
\end{equation}

Next, by using the fact that $E_{V}(f)=E_{0}(Q)$ (cf. \eqref{OAss1}), it is clear that $\|f\|^{4}_{L^{4}}=\|Q\|^{4}_{L^{4}}-2\delta(f)$.
But then, since $M(f)=M(Q)$, it follows 
\begin{align*}
	\|f\|_{\dot{H}^{1}_{V}}^{2}-\( \tfrac{\|f\|_{L^{4}}^{4}}{C_{GN}\| f\|_{L^{2}}}\)^{\frac{2}{3}}
	=\|Q\|_{\dot{H}^{1}}^{2}-\delta(f)-\(\tfrac{\|Q\|^{4}_{L^{4}}-2\delta(f)}{C_{GN}\|Q\|_{L^{2}}}\)^{\frac{2}{3}}.
	\end{align*}
On the other hand, Taylor expansion and \eqref{IDbG} implies 
\begin{align*}
\(\tfrac{\|Q\|^{4}_{L^{4}}-2\delta(f)}{C_{GN}\|Q\|_{L^{2}}}\)^{\frac{2}{3}}
&=\(\|Q\|^{\frac{8}{3}}_{L^{4}}-\tfrac{4}{3}\|Q\|^{-\frac{4}{3}}_{L^{4}}\delta(f)+O(|\delta(f)|^{2})\)
[C_{GN} \|Q\|_{L^{2}}]^{-\frac{2}{3}}\\
&=\|Q\|_{\dot{H}^{1}}^{2}-\delta(f)+O(|\delta(f)|^{2}).
\end{align*}
Thus, combining identities above we obtain
\[
\|f\|_{\dot{H}^{1}_{V}}^{2}-\( \tfrac{\|f\|_{L^{4}}^{4}}{C_{GN}\| f\|_{L^{2}}}\)^{\frac{2}{3}}
=O(|\delta(f)|^{2}),
\]
hence, by \eqref{BoundI} we obtain 
\[
\(\IM \int_{\R^{3}} x\cdot \nabla f \overline{f} dx\)^{2}
\leq C
|\delta(f)|^{2}
\int_{\R^{3}}|x|^{2}|f|^{2}dx.
\]
This completes the proof of lemma.
\end{proof}

\section{Compactness properties}\label{S:Compactness}

\begin{proposition}\label{Criticalsolution}Fix $a>0$ and $1<\mu<2$.
Suppose Theorem \ref{Th2}~(i) fails for some $a>0$. Then we can find a forward global 
solution $u\in C([0, \infty); H^{1}(\R^{3}))$ to \eqref{NLS}  which satisfies
\begin{align}\label{Su}
	&E_{V}(u_{0})=E_{0}(Q), \quad M(u_{0})=M(Q)
\quad \text{and}\quad 
P_{V}(u_{0})\geq 0,\\ \label{Blowup}
	&	\|u \|_{L_{t,x}^{5}([0, \infty)\times\R^{3})}=\infty,
\end{align}
and there exists a  function $x_{0}:[0, \infty) \to \R$ so that
$\left\{u(t, \cdot+x_{0}(t)):t\in [0, \infty)\right\}$ is pre-compact in $H^{1}(\R^{3})$.
\end{proposition}

Before the proof of Proposition~\ref{Criticalsolution}, we need the following lemma.

\begin{lemma}\label{Newconditions}
Suppose that Theorem \ref{Th2}~(i) holds for any $a>0$ with the condition \eqref{Thres} replaced by \eqref{Su}.
Then we can prove the same conclusion in Theorem \ref{Th2}~(i) (for any $a>0$) with the original condition \eqref{Thres}.
\end{lemma}
\begin{proof}
Let $a>0$. Suppose that $u_{0}\in H^{1}(\R^{3})$ such that 
\[
E_{V}(u_{0})M(u_{0})=E_{0}(Q)M(Q)
\quad\text{and}\quad P_{V}(u_{0})\geq 0.
\]
and assume that Theorem \ref{Th2}~(i) is true with the condition \eqref{Su}. 

Writing $\lambda=\frac{M(u_{0})}{M(Q)}$, $v_{0}(x):=\lambda u_{0}(\lambda x)$, $v(t,x)=\lambda u(\lambda^{2}t, \lambda x)$ and
\[
V_{\lambda}(x)=\lambda^{2}V(\lambda x)=\lambda^{2-\mu}a |x|^{-\mu}
\]
we obtain from \eqref{Thres},
\[
E_{V_{\lambda}}(v_{0})=E_{0}(Q), \quad M(v_{0})=M(Q)
\quad \text{and}\quad 
P_{V_{\lambda}}(v_{0})=\lambda P_{V}(u_{0})\geq 0.
\]
Notice also that the function $v$ satisfies
\[
i\partial_{t}v+\Delta v-\lambda^{2-\mu}a |x|^{-\mu}v +|v|^{2}v=0.
\]
Since $\lambda^{2-\mu}a>0$, by  hypothesis we infer $v\in L_{t,x}^{5}(\R\times \R^{3})$, which implies that
$u\in L_{t,x}^{5}(\R\times \R^{3})$. Therefore,  we obtain that $u$ scatters in $H^{1}(\R^{3})$.
\end{proof}

\begin{proof}[{Proof of Proposition \ref{Criticalsolution}}]
 We follow the outline of \cite[Proposition 3.1]{MiaMurphyZheng2021}.
Suppose that Theorem \ref{Th2}~(i) fails. Lemma \ref{Newconditions} implies that there exists  $u_{0}\in H^{1}(\R^{3})$ so that
\[
E_{V}(u_{0})=E_{0}(Q), \quad M(u_{0})=M(Q)
\quad \text{and}\quad 
P_{V}(u_{0})\geq 0,
\] 
Moreover, 
\[
\|u \|_{L_{t,x}^{5}([0, \infty)\times\R^{3})}=\infty,
\]
where $u$ is the corresponding forward-global solution to  \eqref{NLS} with initial data $u_{0}$. By Lemma \ref{GlobalS} we see that $\|u(t)\|_{H^{1}}\lesssim_{Q} 1$ for all $t\in \R$. 
Now we show that there exists a  parameter $x_{0}:[0, \infty) \to \R$ such that
$\left\{u(t, \cdot+x_{0}(t)):t\in [0, \infty)\right\}$ is precompact in $H^{1}(\R^{3})$. 

By \cite[Subsection 3.2]{MiaMurphyZheng2021}, it is enough to show that if  $\{\tau_n\}_{n\in \N}$ is an arbitrary sequence so that $\tau_n \to \infty$, then there exists a sequence $\{x_n\}_{n\in \N}\subset \R^{3}$ such that $u(\tau_n, x+x_n)$ converges strongly in $H^1(\mathbb{R}^{3})$.

Linear profile decomposition (cf. Lemma \ref{LPD}), implies, up to subsequence, that
\[
	u_{n}=u(\tau_n)=\sum_{j=1}^{J} e^{i t_{n}^j H} \tau_{x_n^j} \psi^j +R_n^J.
\]
and the properties in the statement hold. We set $\psi_n^j := e^{i t_{n}^j H} \tau_{x_n^j} \psi^j$. 
 
We claim that $J^{\ast}=1$. Indeed, first assume $J^{\ast}=0$. By the profile decomposition (cf. Lemma \ref{LPD}) we get $\|e^{-it H} u(\tau_n)\|_{L_{t,x}^{5}(\R\times\R^{3})} \to 0$ as $n \to \infty$. But then, from stability (cf. Lemma \ref{StabilityNLS}), we get $\|u\|_{L_{t,x}^5([\tau_n,\infty)\times\R^{3})} \lesssim 1$ for large $n \in \mathbb{N}$, which is a contradiction with the definition of $u$. 

Next, suppose $J^{\ast}\geq 2$. By Lemma~\ref{LPD} we have the following for any $0\leq J \leq J^{\ast}$,

\begin{align*}
	&\lim_{n \to \infty} \big(\sum_{j=1}^{J} M(\psi_n^j) + M(R_n^J)\big) = \lim_{n \to \infty} M(u_{n}) =M(u_0) =M(Q),
	\\
	&\lim_{n \to \infty} \big(\sum_{j=1}^{J} E_{V}(\psi_n^j) + E_{V}(R_n^J)\big) = \lim_{n \to \infty} E_{V}(u_{n}) =E_{V}(u_0) =E_{0}(Q),\\
	&\limsup_{n \to \infty} \big(\sum_{j=1}^{J} \| \psi_n^j\|_{\dot{H}^{1}_{V}}^2 + \| R_n^J\|_{\dot{H}^{1}_{V}}^2\big)
	= \limsup_{n \to \infty} \|u(\tau_n)\|_{\dot{H}^{1}_{V}}^2 \leq \| Q\|_{\dot{H}^{1}}^2.
\end{align*}
Here we also have used Lemma~\ref{GlobalS} in the last inequality.
In particular,
\begin{align}\label{LimS}
	&\lim_{n \to \infty} \big(\sum_{j=1}^{J} S_{V}(\psi_n^j) + S_{V}(R_n^J)\big) = \lim_{n \to \infty} S_{V}(u_{n})
	=S_{V}(u_0) = S_{0}(Q).
\end{align}
It is not hard to show that $\liminf_{n\to \infty}E_{V}(\psi_n^j)>0$  (see, e.g, \cite[(3.13)]{MiaMurphyZheng2021}).
In particular, $\liminf_{n\to \infty}S_{V}(\psi_n^j)> \|\psi^j\|^{2}_{L^{2}}>0$.
Thus, as $J^{\ast}\geq 2$, there exists  $\delta>0$ so that
\begin{align}\label{MEC}
	M(\psi_n^j)E_{V}(\psi_n^j)&\leq M(Q)E_{0}(Q)-\delta,\\\label{GL2C}
	\|\psi_n^j\|_{L^{2}}\|\psi_n^j\|_{\dot{H}^{1}_{V}}& \leq \|Q\|_{L^{2}}\|Q\|_{\dot{H}^{1}}-\delta,\\\label{SVC}
	S_{V}(\psi_n^j)&\leq S_{0}(Q)-\delta,
\end{align}
for sufficiently large $n$.

We will use $\psi_n^j$ to build approximate solutions to \eqref{NLS} under three cases: 
$x^{j}_{n}\equiv 0$ and $t^{j}_{n}\equiv 0$; $x^{j}_{n}\equiv 0$ and $t^{j}_{n}\to \pm \infty$; and $|x^{j}_{n}| \to +\infty$.
For $j$ such that $x^{j}_{n}\equiv 0$ and $t^{j}_{n}\equiv 0$, thanks to \eqref{MEC}-\eqref{GL2C} we infer that $P_{V}(\psi_n^j)\geq 0$
for large $n$ (cf. \cite[Theorem 1.10]{HamanoIkeda20}). Then,  we  can apply Theorem~\ref{Th1} to constitute a global solution obeying global space-time bounds. For $j$ such that $x^{j}_{n}\equiv 0$ and $t^{j}_{n}\to \pm \infty$, by Remark~\ref{Exitencescat}, there exists $v^{j}$ solution to \eqref{NLS} that scatters to $e^{-itH}\psi^{j}$ as  $t\to \pm \infty$. Again, from 
\eqref{MEC}-\eqref{GL2C} we have that the solution is global and satisfies uniform space-time bounds. In either case, we set
\[
v^{j}_{n}(t,x)=v^{j}(t+t^{j}_{n}, x)
\]

Finally, for  $j$ such that $|x^{j}_{n}| \to +\infty$, we have that (cf. \cite[Lemma 2.7 and (2.20)]{GuoWangYao2018}) 
\[
\lim_{n\to \infty}\|\psi_n^j\|_{{H}^{1}_{V}}^{2}=\|\psi^j\|_{{H}^{1}}^{2}>0.
\]
In particular,
\begin{align}\label{MEC11}
	M(\psi^j)E_{0}(\psi^j)\leq M(Q)E_{0}(Q)-\delta,
\end{align}
 Next, if $t^{j}_{n}\equiv 0$, we get
\[
\|\psi^j\|_{L^{2}}\|\psi^j\|_{\dot{H}^{1}}\leq \|Q\|_{L^{2}}\|Q\|_{\dot{H}^{1}}-\delta,
\]
which implies, by \eqref{MEC11}, that $P_{0}(\psi^j)\geq 0$ for $n$ large (cf. \cite[p. 636]{AkahoriNawa2013}).
Moreover, in this case $t^{j}_{n}\equiv 0$, we also have (cf. \eqref{SVC})
\[
S_{0}(\psi^j)= \lim_{n\to \infty}S_{V}(\psi_{n}^j)\leq S_{0}(Q)-\delta.
\]
Therefore, when $t^{j}_{n}\equiv 0$, we obtain that $\psi^j$ satisfies the condition \eqref{PriC}.

On the other hand, if $t^{j}_{n} \to \pm\infty$, we get
\[
\tfrac{1}{2}\|\psi^j\|_{{H}^{1}}^{2}=\lim_{n\to \infty}S_{V}(\psi_{n}^j)\leq S_{0}(Q)-\delta,
\]
where we have used that the nonlinear part of $S_{V}$ tends to zero as $n\to \infty$ (cf. \cite[Lemma 2.8]{GuoWangYao2018}).
This implies that $\psi^j$ satisfies the condition \eqref{seC}. Thus, by Lemma~\ref{P:embedding} we obtain a solution $v_{n}^{j}$
 to \eqref{NLS} with $v_{n}^{j}(0)=\psi_{n}^j$ obeying the global space-time bounds.

Now the idea of the proof is approximate
\[\text{NLS}_{V}(t)u_{n}\approx \sum^{J}_{j=1}v^{j}_{n}(t)+e^{-it H}R^{J}_{n},\]
under tree cases $x^{j}_{n}\equiv 0$ and $t^{j}_{n}\equiv 0$; $x^{j}_{n}\equiv 0$ and $t^{j}_{n}\to \pm \infty$ and $|x^{j}_{n}| \to +\infty$, and we use perturbation argument (cf. Lemma \ref{StabilityNLS}) to obtain a contradiction to \eqref{Blowup}.
With this in mind, we set
\[u^{J}_{n}(t,x):=\sum^{J}_{j=1}v^{j}_{n}(t,x)+e^{-it H}R^{J}_{n}.
\]
First, we note for each $J$,
\begin{equation}\label{Aps}
\|u^{J}_{n}(0)-u_{n}\|_{H^{1}_{x}}\rightarrow0,\quad \text{as $n\rightarrow\infty$}.
\end{equation}
Moreover,  by using the same argument to \cite[see proof of (3.15)-(3.16)]{MiaMurphyZheng2021} we have
\begin{align*}
	&\sup_{J}\limsup_{n\rightarrow\infty}[\|u^{J}_{n}(0)\|_{H^{1}}+ \| {u}^{J}_{n}  \|_{L^{5}_{t,x}(\R\times \R^{3})}] \lesssim_{\delta} 1\\
	&\sup_{J}\limsup_{n\rightarrow\infty}\||\nabla|^{\frac{1}{2}}[(i\partial_{t}-H){u}^{J}_{n}+|{u}^{J}_{n}|^{2}{u}^{J}_{n}]\|_{N(\R)}=0.
\end{align*}
By estimates above and \eqref{Aps}, Lemma~\ref{StabilityNLS} implies that $u\in L^{5}_{t,x}(\R\times \R^{3})$, which is a contradiction to
\eqref{Blowup}.

Therefore $J^{\ast}=1$. In particular, we obtain
\begin{align*}
	u(\tau_n) =e^{it_n H} \tau_{x_n} \psi +R_n
\end{align*}
with $\lim_{n \to \infty}\|R_n\|_{H^1} =0$. Notice that if $|t_n| \to \infty$, then we have a contradiction to the non-scattering of $u$ by the standard argument. Thus,  $u(\tau_n, x+x_n)=\psi(x) +R_n (x+x_n)$, hence $u(\tau_n, \cdot +x_n)$ strongly converges to $\psi$ in $H^1(\mathbb{R}^{3})$. This completes the proof of proposition.

\end{proof}

\section{Modulation analysis}\label{S:Modulation}

Through this section, we assume that $u(t)$ is a solution to \eqref{NLS} with
\begin{equation}\label{22condition}
E_{V}(u_{0})=E_{0}(Q) \qtq{and} \quad M(u_{0})=M(Q).
\end{equation}

For $\delta_{0}>0$ small, we define (Recall that $\delta(t):=\delta(u(t))$) 
\[
I_{0}=\left\{t\in [0, \infty):|\delta(u(t))|<\delta_{0}\quad\text{for $t$ in the domain existence of $u$}\right\},
\] 
where $u(t)$ is the corresponding solution to Cauchy problem \eqref{NLS}.

\begin{lemma}\label{LemmaMod}  
For any $\epsilon>0$, there exists $\delta_{0}=\delta_{0}(\epsilon)>0$ small such that if
$|\delta(u(t))|<\delta_{0}$, then there exists $(\theta_{0}(t), y_{0}(t))\in \R\times\R^{3}$ so that
\begin{equation}\label{CondiModu}
\| u(t)-e^{i\theta_{0}(t)}Q(\cdot-y_{0}(t))\|_{H^{1}}<\epsilon.
\end{equation}
\end{lemma}
\begin{proof}
We argue by contradiction. Thus, suppose that there exist $\epsilon>0$ small and a sequence of times
$\left\{t_{n}\right\}\subset \R$ with
\begin{equation}\label{Contra11}
|\delta(u(t_{n}))|\to 0, \quad
\inf_{\theta\in \R}\inf_{y\in \R^{3}}\|u(t_{n})-e^{i\theta}Q(\cdot-y)\|_{H^{1}}\geq \epsilon.
\end{equation}
By using \eqref{22condition} we see that (recall that $|\delta(u(t_{n}))|\to 0$)
\[
S_{V}(u(t_{n}))= S_{0}(Q)\quad \text{and}\quad  N_{V}(u(t_{n}))\to N_{0}(Q)=0,
\]
where $N_{V}$ is the Nehari functional,
\[
N_{V}(f)=\|\nabla f\|^{2}_{L^{2}}+\|f\|^{2}_{L^{2}}+ \int_{\R^{3}}V(x)|f|^{2}dx -\|f\|_{L^{4}}^{4}
\quad \text{for $f\in H^{1}(\R^{3})$}.
\]
But then, we infer that $N_{0}(u(t_{n})) \leq 0$ for $n$ sufficiently large, which implies that
$\left\{u(t_{n})\right\}$ is a minimizing sequence of problem 
\[
S_0(Q)=\inf\left\{S_{0}(f): f\in H^{1}(\R^{3})\setminus\left\{0\right\},  N_{0}(f) \leq 0\right\}.
\]
From \cite[Proposition 3.12]{SLC} we have that there exists $(\theta_{n}, y_{n})\in \R^{4}$ such that
$e^{i\theta_{n}}u(t_{n},\cdot+y_{n})\to Q$ in $H^{1}(\R^{3})$. However, this leaves a contradiction to \eqref{Contra11}.
\end{proof}

\begin{remark}\label{Yinfinity}
Let $R\geq1$. If $\delta_{0}$ is sufficiently small in Lemma \ref{LemmaMod} we can assume that
\begin{equation}\label{Nobound}
|y_{0}(t)|\geq R\quad \text{for $t\in \R$}.
\end{equation}
Indeed, if \eqref{Nobound} is false, then there exists a sequence  $\left\{t_{n}\right\}$
such that
\begin{equation}\label{Bounddelta}
\text{ $|\delta(t_{n})|\to0$ and $|y_{0}(t_{n})|\leq R$ for all $n\in \N$}.
\end{equation}
Moreover, by using  \eqref{CondiModu} we see that (see proof of Lemma~\ref{LemmaMod})
\begin{align}\label{Cv112}
	e^{-i \theta_{0}(t_{n})}u(t_{n}, \cdot+y_{0}(t_{n}))\to Q \qtq{in} H^{1}(\R^{3}).
\end{align}
 In particular,  we get $E_{V}(u(t_{n}))=E_{0}(Q)=\lim_{n\to \infty}E_{0}(u(t_{n})) $, which implies
by \eqref{Bounddelta} we get
\[
\int_{\R^{3}}V(x)|u(t_{n},x)|^{2}dx\to 0 \qtq{as $n\to \infty$.}
\]
But then, again by \eqref{Cv112} we have
\[
\int_{\R^{3}}V(x)|Q(\cdot-y_{0}(t_{n}))|^{2}dx\to 0 \qtq{as $n\to \infty$,}
\]
which is a contradiction because the sequence $\left\{y_{0}(t_{n})\right\}$ is bounded.
\end{remark}

By Lemma \ref{LemmaMod} and an application of implicit function theorem we obtain the following result.

\begin{lemma}\label{ExistenceF}
If  $\delta_{0}>0$ is sufficiently small, then there exist two functions $\theta: I_{0}\to \R$ and $y: I_{0}\to \R^{3}$ so that 
\begin{equation}\label{Taylor}
\| u(t)-e^{i\theta(t)}Q(\cdot-y(t))\|_{H^{1}} \ll 1.
\end{equation}
Writing $g(t):=g_{1}(t)+i g_{2}(t)=e^{-i\theta(t)}[u(t)-e^{i\theta(t)}Q(\cdot-y(t))]$, we have that $g$ satisfies
\begin{equation}\label{Ortogonality}
\< g_{2}(t), Q(\cdot-y(t)) \>=\< g_{1}(t), \partial_{x_j}Q(\cdot-y(t))\>\equiv 0 \quad (j=1,2,3).
\end{equation}
\end{lemma}
\begin{proof}
The proof is the same as in \cite[Lemma 5.3]{MiaMurphyZheng2021}.
\end{proof}

\begin{proposition}[Modulation]\label{Modilation11}
Fix $a>0$ and $0<\mu<2$. Suppose that $u(t)$ is a solution to \eqref{NLS} obeying 
\eqref{22condition} and $|\delta(0)|=|\delta(u_{0})|>0$. Then, there exist 
$\delta_{0}>0$ sufficiently small and two functions $\theta: I_{0}\to \R$ and $y: I_{0}\to \R^3$ so that $u(t)$
admits the decomposition
\begin{equation}\label{DecomU}
u(t,x)=e^{i \theta(t)}[g(t)+Q(x-y(t))]\quad \text{for all $t\in I_{0}$},
\end{equation}
and the following holds:
\begin{equation}\label{Estimatemodu}
\frac{e^{-2|y(t)|}}{|y(t)|^{2}}+|y^{\prime}(t)|+
\left[\int_{\R^{3}}V(x)|u(t,x)|^{2}dx\right]^\frac{1}{2}
\lesssim |\delta(t)|\sim \|g(t)\|_{H^{1}} \quad \text{for all $t\in I_{0}$}.
\end{equation}
Furthermore, letting $g=\alpha Q(\cdot-y)+h$ and $g=g_{1}+ig_{2}$, where
\[
\alpha=\frac{\< g_{1}(\cdot{{+}}y),\Delta Q \>}{\<Q, \Delta Q\>}\in \R,
\]
and  we have 
\begin{align}\label{alfaE}
	|\alpha(t)|&\sim |\delta(t)|, \quad \|h(t)\|_{H^{1}}\sim |\delta(t)|  \qtq{and}\\\label{ultima}
	|\alpha^{\prime}(t)|&\lesssim|\delta(t)| \quad \text{for $t\in I_{0}$}.
\end{align}
\end{proposition}
\begin{proof}
With Lemma~\ref{ExistenceF} and Remark~\ref{Yinfinity} in hand, the proof of \eqref{Estimatemodu} and \eqref{alfaE} is essentially the same as in \cite[Proposition 5.1]{MiaMurphyZheng2021}. 

The proof of estimate \eqref{ultima} is similar to that given in \cite[Lemma 4.3]{DuyckaertsRou2010} (see also \cite[Lemma 5.6]{MiaMurphyZheng2021}). Indeed, by \eqref{DecomU} we have
\[
h(t,x)=e^{-i\theta(t)}[u(t)-e^{i\theta(t)}(1+\alpha(t))Q(x-y(t))].
\]
Let $h=h_{1}+ih_{2}$. Then we have the following orthogonality relations 
(see proof of Lemma 5.4 in \cite{MiaMurphyZheng2021}),
\begin{equation}\label{OOR}
\<h_{1}, \Delta Q(\cdot-y) \>=\<h_{2}, Q(\cdot-y) \>=\<h_{1}, \partial_{j} Q(\cdot-y) \>=0
\end{equation}
for $j=1$, $2$, $3$. In particular, by \eqref{Estimatemodu}-\eqref{alfaE} we get 
$\<\partial_{ t}h_{1},  Q(\cdot-y) \>\lesssim |\delta(t)|$. 

Now, using the equation \eqref{NLS} and \eqref{ELS} we derive the equation
\begin{equation}\label{EFM}
\begin{split}
i\partial_{t}h+\Delta h-\theta^{\prime}h-Ve^{-i \theta}u
-\theta^{\prime}(1+\alpha)Q(x-y)+i\alpha^{\prime}Q(x-y)\\
+i(1+\alpha)y^{\prime}\cdot\nabla Q(x-y)+(1+\alpha)Q(x-y)
+f(e^{-i \theta}u)-(1+\alpha)f(Q(x-y))=0,
\end{split}
\end{equation}
where $f(z)=z|z|^{2}$. Then, multiplying Eq. \eqref{EFM} with $Q(\cdot-y)$, taking integral and imaginary part, by  estimates \eqref{Estimatemodu}-\eqref{alfaE}, it is not difficult to show that (see proof of Lemma 5.6 in \cite{MiaMurphyZheng2021})
\[
|\alpha^{\prime}(t)|\lesssim |\delta(t)|\quad \text{for all $t\in I_{0}$},
\]
which completes the proof of lemma.
\end{proof}

\section{Precluding the compact solution}\label{S:Impossibility}

Throughout this section we assume that $u$ is the solution constructed in Proposition \ref{Criticalsolution}.
In particular, $u$ satisfies \eqref{22condition}, $P_{V}(u_{0})\geq 0$ and
\[
\|u \|_{L_{t, x}^{5}([0, \infty)\times\R^{3})}=\infty.
\]
Moreover, 
\[
\text{$\left\{u(t, \cdot+x_{0}(t)):t\in [0, \infty)\right\}$ is pre-compact in $H^{1}(\R^{3})$.}
\]
From \eqref{GINequ} we see that 
\begin{equation}\label{deltapositive}
\delta(t):=\delta(u(t))>0 \quad \text{for all $t\in [0, \infty)$}.
\end{equation}

\begin{lemma}\label{Parametrization}
If $\delta_{0}$ is  small, then there exists a constant $C>0$ so that
\[
|x_{0}(t)-y(t)|<C \quad \text{for  $t\in I_{0}$}.
\]
Here, the parameter $y(t)$ is given in Proposition \ref{Modilation11}.
\end{lemma}
\begin{proof}
The proof is the same as the proof of \cite[Lemma 4.2]{MiaMurphyZheng2021}.
\end{proof}

From Lemma \ref{Parametrization} we infer that
\begin{equation}\label{CompacNew}
\text{$\left\{u(t, \cdot+x(t))\right\}$ is pre-compact in $H^{1}(\R^{3})$},
\end{equation}
where the spatial center $x(t)$ is given by
\[x(t)=
\begin{cases}
x_{0}(t)& \quad t\in [0, \infty)\setminus I_{0},\\
y(t)&\quad t\in I_{0}.
\end{cases}
\]

\begin{proposition}\label{Bounded11}
If the spacial center $x(t)$ is bounded, then $x(t)$ is unbounded.
\end{proposition}

Proposition~\ref{Bounded11} will be a consequence of the following lemmas.

\begin{lemma}\label{SequenceInf}
For any time sequence $\left\{t_{n}\right\}\subset [0, \infty)$, we have
\begin{equation}\label{ZeroPoten}
|x(t_{n})|\to \infty \quad \text{if and only if}\quad \int_{\R^{3}}V(x)|u(t_{n},x)|^{2}dx\to 0.
\end{equation}
\end{lemma}
\begin{proof}
With Lemma~\ref{GHI} in hand, the proof is the same as in \cite[Lemma 4.{{3}}]{MiaMurphyZheng2021}.
\end{proof}

\begin{lemma}\label{DeltaZero}
Suppose  $t_{n}\to \infty$. Then
\begin{equation}\label{equivalence}
|x(t_{n})|\to \infty \quad \text{if and only if}\quad \delta(t_{n})\to 0.
\end{equation}
\end{lemma}
\begin{proof}
If $\delta(t_{n})\to 0$, then combining  \eqref{ZeroPoten} and estimate \eqref{Estimatemodu}
we see that $|x(t_{n})|\to \infty$.

Next, let $t_{n}\to \infty$ and assume by contradiction that $|x(t_{n})|\to \infty$ but, 
possibly for a subsequence only,
\begin{equation}\label{Zeroplus}
\delta(u(t_{n}))\geq c>0.
\end{equation}

As $\left\{u(t_{n}, \cdot+x(t_{n}))\right\}$ is pre-compact in $H^{1}(\R^{3})$, we have that there exists
$v_{0}\in H^{1}(\R^{3})$ so that 
\begin{equation}\label{CompactV}
u(t_{n}, \cdot+x(t_{n})) \to v_{0} \quad \text{in}\quad H^{1}(\R^{3}),
\end{equation}
along some subsequence in $n$. In particular, since $|x(t_{n})|\to \infty$, it follows from \eqref{Zeroplus} and \eqref{ZeroPoten}, 
\[
M(v_{0})=M(Q), \quad E_{0}(v_{0})=E_{0}(Q) \quad \text{and}\quad \|\nabla v_{0}\|^{2}_{L^{2}}<\|\nabla Q\|^{2}_{L^{2}}.
\]
An application of \cite[Theorem 3]{DuyckaertsRou2010}  implies that the solution $v$ of the free NLS on $\R^{3}$ (i.e., \eqref{NLS} with $a=0$) with initial data  $v_{0}$ is global and either scatters as $t\to \infty$ or as $t\to -\infty$ (or both). 

Suppose that $v$ scatters as $t \to \infty$. As $|x(t_{n})|\to \infty$,  we can use a similar 
argument  as in  \cite[Lemma 4.4]{MiaMurphyZheng2021} to find  a solution $v_{n}$ to \eqref{NLS} so that
\[
v_{n}(0)= v_0(\cdot -x(t_n)) \qtq{and} \|v_{n}\|_{L_{t,x}^5([0,\infty)\times\R^{3})}\lesssim 1
\]
for large $n$. Notice that by \eqref{CompactV} we get $\| u(t_n,x) - {v}_n(0)\|_{H^1} \to 0$ as $t\to \infty$.
Then the stability result (cf. Lemma~\ref{StabilityNLS}) applies and  
\[
\|u(t_n+t)\|_{L_{t,x}^5([0,\infty)\times\R^{3})}
 = \|u\|_{L_{t,x}^5([t_n,\infty)\times\R)} \lesssim 1
\]
for large $n$, which contradicts that the $L_{t,x}^5$-norm of $u$ is infinite.

Next, uppose that $v$ scatters as $t \to -\infty$. An argument similar to the one developed above shows that
\[\|u(t_n+t)\|_{L_{t,x}^5((-\infty,0]\times\R)} = \|u\|_{L_{t,x}^5((-\infty,t_n]\times\R)} \lesssim 1
\] for large $n$. This also contradicts that $u$ does not scatter. 
Therefore, $\delta(u(t_n)) \to 0$ as $n \to \infty$. This completes the proof of lemma.
\end{proof}

Recall that {{$F_{\infty,0}$}} is defined in Lemma \ref{VirialIden}. We have the following result.
\begin{lemma}\label{Compa}
Fix $a>0$.
There exists $c>0$ so that
\begin{equation}\label{InequeN}
{{F_{\infty,0}}}[u(t)]=8\|\nabla u(t)\|^{2}_{L^{2}}-6\|u(t)\|^{4}_{L^{4}}\geq c\delta(t).
\end{equation}
\end{lemma}
\begin{proof}
Suppose \eqref{InequeN} is false. Then there exists  $\left\{t_{n}\right\}_{n\in\N}$ such that
\begin{equation}\label{Contra22}
{{F_{\infty,0}}}[u(t_{n})]\leq \tfrac{1}{n}\delta(t_{n}).
\end{equation}
Notice that $\left\{\delta(t_{n})\right\}$ is bounded (cf. Lemma~\ref{GlobalS}). We claim that
\[
\delta(t_{n})\to 0 \quad \text{as $n\to\infty$}.
\]
Indeed, by using the Pohozaev's identities and $E_{V}(u(t_{n}))=E_{0}(Q)$ we have (cf. \eqref{PoQ})
\begin{equation}\label{Pohi22}
\|  Q \|^{4}_{L^{4}}-\|u(t_{n}) \|^{4}_{L^{4}}=2\delta({t_n}).
\end{equation}
Thus, as ${{F_{\infty,0}}}[Q]=0$, we get
\begin{equation}\label{Aux11}
{{F_{\infty,0}}}[u(t_{n})]
=4\delta(t_{n})-8\int_{\R^{3}}V(x)|u(t_{n}, x)|^{2}dx.
\end{equation}
By using sharp Gagliardo-Nirenberg inequality \eqref{GI}, \eqref{PoQ} and \eqref{GINequ} we deduce  
\begin{align*}
\|u(t_{n})\|^{4}_{L^{4}}&\leq  C_\text{GN}\|u\|_{L^2}\|\nabla u(t_n)\|_{L^2}^3\\
&\leq \frac{\| \nabla u(t_{n})\|_{L^{2}}}{\|\nabla Q\|_{L^{2}}}\cdot\frac{\|  Q \|^{4}_{L^{4}}}{\|\nabla Q\|^{2}_{L^{2}}}
\cdot\| \nabla u(t_{n})\|^{2}_{L^{2}}\\
&<\tfrac{4}{3}\|\nabla u(t_{n})\|^{2}_{L^{2}}.
\end{align*}
Since $F_{\infty,0}[u(t_n)] \rightarrow 0$ as $n \rightarrow \infty$ we have
\begin{align*}
	0
		\leftarrow \|\nabla u(t_n)\|_{L^2}^2\left(\tfrac{4}{3} - C_\text{GN}\|u\|_{L^2}\|\nabla u(t_n)\|_{L^2}\right)
		= \tfrac{4}{3}\|\nabla u(t_n)\|_{L^2}^2\left(1 - \frac{\|\nabla u(t_n)\|_{L^2}}{\|\nabla Q\|_{L^2}}\right)
\end{align*}
as $n \rightarrow \infty$.
Here, it follows from compactness of $u$ that there exists a positive constant $A > 0$ such that $A \cdot M(u_0) \leq \|\nabla u(t)\|_{L^2}^2$ for each $t \in \mathbb{R}$ (e.g. see \cite[Lemma 6.6]{Xu}).
So, we obtain
\begin{align*}
	\|\nabla u(t_n)\|_{L^2}
		\rightarrow \|\nabla Q\|_{L^2}
\end{align*}
as $n \rightarrow \infty$.

In particular,
\begin{equation}\label{Secondelta}
\delta(t_{n})=-\int_{\R^{3}}V(x)|u(t_{n}, x)|^{2}dx+o(1) \quad \text{as $n\to \infty$}.
\end{equation}
Combining \eqref{Aux11} and \eqref{Secondelta} we obtain the claim.

Finally, by Proposition \ref{Modilation11} we have
\[
\int_{\R^{3}}V(x)|u(t_{n}, x)|^{2}dx\lesssim \delta(t_{n})^{2}\leq \tfrac{1}{4}\delta(t_{n})\quad \text{for $n$ large}.
\]
Thus, by using \eqref{Aux11} and \eqref{Contra22} we get
\[
2\delta(t_{n})\leq \tfrac{1}{n}\delta(t_{n})\quad \text{for $n$ large},
\]
which is a contradiction with \eqref{deltapositive}.
\end{proof}

\begin{proof}[{Proof of Proposition \ref{Bounded11}}]
The proof is divided into 2 steps.

\textsl{Step 1.} Virial estimate. Let $T>0$ and $\epsilon>0$, then there exists  $\rho_{\epsilon}=\rho(\epsilon)>0$ such that
\begin{equation}\label{NewVirial}
\int^{T}_{0}\delta(t)dt\lesssim\epsilon T+
[\rho_{\epsilon}+\sup_{t\in [0,T]}|x(t)|]\|u\|_{L^{\infty}_{t}H^{1}_{x}}^{2}.
\end{equation}
The proof of Step 1 is the same as for \cite[Lemma 4.7]{MiaMurphyZheng2021}, using our Lemma~\ref{Compa} instead of Lemma 4.5 of their paper.

\textsl{Step 2.} Conclusion. We argue by contradiction. If the spatial center $x(t)$ is bounded, then by Step 1 above we have
\[
\tfrac{1}{T}\int^{T}_{0}\delta(t)dt\lesssim\epsilon
+\tfrac{1}{T}\rho_{\epsilon}\quad \text{for all $T>0$}, 
\]
and for any $\epsilon>0$. Consider a sequence $\epsilon_{n}\to 0$ as $n\to \infty$. By 
choose appropriately times $T_{n}\to \infty$, an application of the mean value theorem for integrals
implies that there exists a time sequence $t_{n}\to \infty$ 
such that that $\delta(t_{n})\to 0$ as $n\to \infty$ (recall that  $\delta(t)> 0$ for $t\geq 0$).
But then Lemma \ref{DeltaZero} implies that $|x(t_{n})|\to \infty$, which is a contradiction. 
This completes the proof of proposition.
\end{proof}

\begin{proposition}\label{NaoBounded}
If the spacial center $x(t)$ is unbounded, then $x(t)$ is bounded.
\end{proposition}
\begin{proof}
The proof is essentially identical to that of \cite[Proposition 4.8]{MiaMurphyZheng2021} and we omit the details.
\end{proof}

Now we are ready to give the proof of scattering result of Theorem~\ref{Th2}.

\begin{proof}[{Proof of Theorem~\ref{Th2}}~(i)]
If Theorem \ref{Th2}~(i)  is not true, then there exists a critical element 
$u_{0}\in H^{1}(\R^{3})$ and a spatial center $x(t)$ such that the corresponding solution to \eqref{NLS} satisfies
$\left\{u(t, \cdot+ x(t)): t\geq 0\right\}$ is precompact in $H^{1}(\R^{3})$ (cf.  {{Proposition \ref{Criticalsolution}}}). However,
we have that this is impossible by Propositions \ref{Bounded11} and \ref{NaoBounded}.
\end{proof}

\section{Criteria for Blow-up}\label{Sec22}
In this section we give the proof of the blow-up result of Theorem~\ref{Th2}. Before the proof of Theorem~\ref{Th2}~(ii), we need the following result.
\begin{proposition}\label{Pro11}
If the initial data $u_{0}\in H^{1}(\R^{3})$ satisfies $|x|u_{0}\in L^{2}(\R^{3})$,
\begin{equation}\label{NewH}
M(u_{0})=M(Q), \quad E_{V}(u_{0})=E_{0}(Q),
\quad \text{and}\quad 
P_{V}(u_{0})<0,
\end{equation}
then the solution $u$ to \eqref{NLS} with data $u_{0}$ blows up in both directions.
\end{proposition}

To prove the proposition above we need some preparation. We begin with the following lemma.

\begin{lemma}\label{Comparative}
Fix $a>0$ and $0<\mu<2$.
Under assumption of Proposition~\ref{Pro11} we have 
\begin{align}\label{GNP}
\(\IM \int_{\R^{3}} x\cdot \nabla u(t) \overline{u(t)} dx\)^{2}
\lesssim
|P_{V}(u(t))|^{2}
\int_{\R^{3}}|x|^{2}|u(t)|^{2}dx,
\end{align}
where $u(t)$ is the corresponding solution to \eqref{NLS} with  data $u_{0}$.
\end{lemma}
\begin{proof}
As $P_{V}(u_{0})<0$,  from Lemma~\ref{CondiBlo} we see that $\delta(t)<0$ and $P_{V}(u(t))<0$ for all $t$ in the existence time.
Now, since
\[
\int_{\R^{3}}|\nabla u|^{2}-\tfrac{3}{4}|u|^{4}dx=3E_{V}(u)-\int_{\R^{3}}\tfrac{1}{2}|\nabla u|^{2}+\tfrac{3}{2}V(x)|u|^{2}dx
\]
and $E_{V}(u)=E_{0}(Q)$, by \eqref{PoQ} we  obtain
\begin{align*}
	P_{V}(u(t))&=2\(\tfrac{1}{2}\delta(t)-\int_{\R^{3}}V(x)| u(t,x)|^{2}dx\) -\int_{\R^{3}} x\cdot\nabla V(x)|u(t,x)|^{2}dx\\
	&=\delta(t)-(2-\mu)\int_{\R^{3}}V(x)| u(t,x)|^{2}dx\leq \delta(t)<0
\end{align*}
for all $t$ in the existence time. Thus, $|\delta(t)|^{2}\leq |P_{V}(u(t))|^{2}$ and  Lemma~\ref{OtherGN} implies \eqref{GNP}.
\end{proof}

\begin{lemma}\label{LP11}
Under assumption of Proposition~\ref{Pro11}, if $u(t)$  is global in positive time, then 
there exist postive constant $C>0$ and $c>0$ such  that
\begin{equation}\label{Decai}
\int^{\infty}_{t}|\delta(\tau)| d\tau \leq Ce^{-ct}\quad \text{for all $t\in (0, \infty)$}.
\end{equation}
\end{lemma}
\begin{proof}
Writing $f(t)=\|xu(t)\|^{2}_{L^{2}}$, we see that (cf. Lemma~\ref{VirialIden})
\[
f^{\prime}(t)=4\IM \int_{\R^{3}} \overline{u}(t)\nabla u(t)\cdot xdx, \quad
f^{\prime\prime}(t)=4P_{V}(u(t)).
\]
Lemma~\ref{CondiBlo} implies that
\[
f^{\prime}(t_{2})-f^{\prime}(t_{1})=\int^{t_{2}}_{t_{1}}f^{\prime\prime}(s)ds=
4\int^{t_{2}}_{t_{1}}P_{V}(u(s))ds<0 \qtq{for $t_{1}<t_{2}$.}
\]
We claim that  $f^{\prime}(t)>0$ for all $t$ in the existence time. Indeed, assume by contradiction that there exists $t^{\ast}\in \R$
so that $f^{\prime}(t^{\ast})\leq 0$. Then inequality above shows that $f^{\prime}(t)<0$ for any $t>t^{\ast}$, which is
a contradiction because $f(t)> 0$ for all $t\in [0, +\infty)$. Therefore,
\begin{align}\label{Nece}
\tfrac{1}{4}f^{\prime}(t)	= \IM \int_{\R^{3}} \overline{u}(t, x)\nabla u(t, x)\cdot xdx>0, \quad \text{for all $t\in (0, \infty)$.}
\end{align}

Now we will show that $f^{\prime}(t)\leq Ce^{-ct}$ for all $t\geq 0$. Indeed, note that $f>0$, $f^{\prime}>0$, and $f^{\prime\prime}<0$.
Thus,  thanks to Lemma~\ref{Comparative} we see that
\[
|f^{\prime}(t)|^{2} \lesssim (f^{\prime\prime}(t))^{2}f(t) \quad \text{for all $t$ in the existence time,}
\]
which implies 
\begin{equation}\label{Axi11}
\frac{f^{\prime}(t)}{\sqrt{f(t)}}\lesssim - f^{\prime\prime}(t)
 \quad \text{for all $t$ in the existence time}.
\end{equation}
Integrating inequality above on $(0, t)$ we get
\[
\sqrt{f(t)}-\sqrt{f(0)}\lesssim -f^{\prime}(t)+f^{\prime}(0)\lesssim f^{\prime}(0),
\]
i.e,  $\sqrt{f(t)}$ is bounded. From \eqref{Axi11}, it follows that  $f^{\prime}(t)\lesssim -f^{\prime\prime}(t)$
for all $t$ in the existence time, which shows $f^{\prime}(t)\leq Ce^{-ct}$ for some constants $C>0$, $c>0$.
 In particular, $\lim_{t\to \infty}f^{\prime}(t)=0$. 

Finally, since $0<-\delta(t)\leq -P_{V}(u(t))$, we get
\begin{align*}
\int^{\infty}_{t}|\delta(s)|ds
& = \int^{\infty}_{t}[-\delta(s)]ds\leq 
\int^{\infty}_{t}[-P_{V}(u(s))]ds \\
&=\tfrac{1}{4}\int^{\infty}_{t}[-f^{\prime\prime}(s)]ds
=\tfrac{1}{4}f^{\prime}(t)\leq Ce^{-ct},
\end{align*}
for $t\in(0, +\infty)$. this completes the proof.
\end{proof}

\begin{proof}[{Proof of Proposition \ref{Pro11}}]
Assume that the initial data $u_{0}\in H^{1}(\R^{3})$ satisfies $|x|u_{0}\in L^{2}(\R^{3})$,
\begin{equation}\label{FinaC}
M(u_{0})=M(Q), \quad E_{V}(u_{0})=E_{0}(Q),
\quad \text{and}\quad 
P_{V}(u_{0})<0.
\end{equation}
From Lemma~\ref{CondiBlo} we see that $\delta(t)<0$ for all $t$ in the existence time.

\textit{Step 1.} \textit{The corresponding solution $u(t)$ to \eqref{NLS} with initial data $u_{0}$ is not global in positive time.}
Indeed, by contradiction, assume that  $u$ global in positive time. 
By \eqref{Decai} we deduce that there exists $\left\{t_{n}\right\}_{n\in \N}$ with $t_{n}\to +\infty$ such that
$\lim_{n \to \infty}\delta(t_{n})=0$. Fix such $\left\{t_{n}\right\}_{n\in \N}$.

Notice that  $\lim_{t \to \infty}\delta(t)=0$. If not, there exists a sequence $\left\{t^{\prime}_{n}\right\}_{n\in \N}$ such that
$-\delta(t^{\prime}_{n})\geq \epsilon$ for some $\epsilon\in(0, \delta_{0})$.
Extracting subsequences of $\left\{t_{n}\right\}_{n\in \N}$ and $\left\{t^{\prime}_{n}\right\}_{n\in \N}$ if necessary, 
we can assume the following properties:
\[
t_{n}<t^{\prime}_{n}, \quad -\delta(t^{\prime}_{n})=\epsilon, \quad -\delta(t)<\epsilon \quad \text{for all $t\in [t_{n}, t^{\prime}_{n})$}.
\]
On $[t_{n}, t^{\prime}_{n})$ the parameter $\alpha(t)$ is well defined. As $|\alpha^{\prime}(t)|\leq C|\delta(t)|$
(cf. \eqref{ultima}), estimate \eqref{Decai} implies 
\begin{equation}\label{limnn}
\lim_{n\to \infty}|\alpha(t_{n})-\alpha(t^{\prime}_{n})|=0.
\end{equation}
Thus, as $|\alpha| \sim |\delta|$ (cf. \eqref{alfaE}), we deduce
\[
|\alpha(t_{n})| \sim |\delta(t_{n})|\to 0 \quad \text{and} \quad 
|\alpha(t^{\prime}_{n})| \sim |\delta(t^{\prime}_{n})|=\epsilon>0,
\]
which is a contradiction to \eqref{limnn}. Therefore, $\lim_{t \to \infty}\delta(t)=0$.
Note that by estimate \eqref{Estimatemodu} we also have
\begin{equation}\label{ContaF}
\frac{e^{-2|y(t)|}}{|y(t)|^{2}}\lesssim |\delta(t)|\to 0 \qtq{as $t\to \infty$.}
\end{equation}

Now,   by using  \eqref{Estimatemodu} and \eqref{Decai} we get  ,
\[
|y(t)-y(t_{1})|\leq \int^{t}_{t_{1}}|y^{\prime}(s)|ds \lesssim
\int^{t}_{t_{1}}|\delta(s)|ds\lesssim e^{-ct_{1}} \qtq{for all $t>t_{1}$}.
\]
This implies that there exists a sequence $\left\{\tau_{n}\right\}_{n\in \N}$ such that $\tau_{n} \to \infty$ with
$|y(\tau_{n})| \to a\in \R$. In particular,
\[
\lim_{n\to \infty}\frac{e^{-2|y(\tau_{n})|}}{|y(\tau_{n})|^{2}} >0,
\]
which is a contradiction to \eqref{ContaF}. Thus, $u(t)$ is not global in positive time.

\textit{Step 2.} \textit{The solution $u$ is not global in negative time.}  Suppose by contradiction that
 $u$ global in negative time. Writing $v(t,x):=\overline{u(-t,x)}$,  we have  that  $v$ is a  global in positive time solution to \eqref{NLS}.

But then, since $|x|v(0)\in L^{2}(\R^{3})$, $E_{V}(v(0))=E_{0}(Q)$, $M(v(0))=M(Q)$ and $P_{V}(v(0))<0$, it follows 
by Step 1 above that $v$ blows-up in positive time, which is a contradiction.

This completes the proof of proposition.
\end{proof}

\begin{proof}[{Proof of Theorem~\ref{Th2}~(ii)}]
First, assume $|x|u_{0}\in L^{2}(\R^{3})$. Then the proof is a direct consequence of Proposition~\ref{Pro11} and the following claim:
\begin{claim}
Assume that Theorem~\ref{Th2}~(ii) holds for $a>0$ with the condition
\eqref{BlowC} replaced by \eqref{NewH}. Then we have the same conclusion in Theorem~\ref{Th2}~(ii) with the 
original hypothesis \eqref{BlowC}. 
\end{claim}
\begin{proof}[{Proof of Claim}]
The proof is very similar to that given in Lemma~\ref{Newconditions}.
Let $a>0$. Assume that Theorem~\ref{Th2}~(ii) is true with the condition \eqref{NewH}. Consider $u_{0}\in H^{1}(\R^{3})$ such that 
$|x|u_{0}\in L^{2}(\R^{3})$,
\[
E_{V}(u_{0})M(u_{0})=E_{0}(Q)M(Q)
\quad\text{and}\quad P_{V}(u_{0})<0.
\]
Writing $V_{\lambda}(x)=\lambda^{2}V(\lambda x)$, $v_{0}(x)=\lambda u_{0}(\lambda x)$ and 
$v(t,x)=\lambda u(\lambda^{2}t, \lambda x)$ with
$\lambda={{\frac{M(u_{0})}{M(Q)}}}$ we deduce
\[
E_{V_{\lambda}}(v_{0})=E_{0}(Q), \quad M(v_{0})=M(Q)
\quad \text{and}\quad  P_{V}(v_{0})=\lambda P_{V}(u_{0})<0.
\]
As $v$ satisfies the equation
\[
i\partial_{t}v+\Delta v -\lambda^{\mu-2}a|x|^{-\mu}v+|v|^{2}v=0
\] 
and $\lambda^{\mu-2}a>0$, Proposition~\ref{Pro11} implies that $v$ blow up in both directions. In particular, we see that
$u$ blow up in both directions. This completes the proof of claim.
\end{proof}

Next, assume that $u_{0}$  is radially symmetric. The proof is based on \cite[Subsection 5.2]{DuyckaertsRou2010}. We consider only positive time. We assume for contradiction that the solution $u$ exists on $[0,\infty)$ under the assumptions of Theorem~\ref{Th2}~(ii).
Then, we prove $|x|u_0 \in L^2(\mathbb{R}^3)$.

We define a functional 
\begin{align*}
	J_R[u(t)]
		:= \int_{\mathbb{R}^3}w_R(x)|u(t,x)|^2dx,
\end{align*}
where $w_R$ is defined by \eqref{WRR}. We also assume $\phi^{\prime\prime}(r) \leq 2$. Then, we have
\begin{align*}
	\frac{d^2}{dt^2}J_R[u(t)]
		& = F_{R,V}[u(t)] \\
		& = 4\delta(t) + \int_{\mathbb{R}^3}(-\Delta\Delta w_R)|u|^2dx - \int_{|x| \geq R}(\Delta w_R - 6)|u|^4dx \\
		& + 4\int_{|x| \geq R}(\phi^{\prime\prime}\(\tfrac{x}{R}\) - 2)|\nabla u|^2dx + 2\int_{\mathbb{R}^3}\left\{\tfrac{\mu R}{|x|}\phi^{\prime}\(\tfrac{x}{R}\) - 4\right\}\frac{a}{|x|^\mu}|u|^2dx \\
		& =: 4\delta(t) + A_R[u(t)] + 2\int_{\mathbb{R}^3}\left\{\tfrac{\mu R}{|x|}\phi^{\prime}\(\tfrac{x}{R}\)  - 4\right\}\frac{a}{|x|^\mu}|u|^2dx.
\end{align*}
We see 
\begin{align*}
\tfrac{\mu R}{|x|}\phi^{\prime}\(\tfrac{x}{R}\) - 4
		\leq 0
\end{align*}
from simple calculation (recall that $1<\mu<2$ and $\phi^{\prime}(r)\leq 2r$).
The argument in \cite[Subsection 5.2]{DuyckaertsRou2010} with Proposition~\ref{Modilation11} deduces that there exists $R_0 > 0$ such that
\begin{align}\label{Eq1}
	A_R[u(t)]
		\leq - 2\delta(t)
\end{align}
for any $R \geq R_0$ and any $t \in [0,\infty)$. In particular, we see that
$\frac{d^2}{dt^2}J_R[u(t)]\leq 2\delta(t)<0$ for any $R \geq R_0$ and any $t \in [0,\infty)$.

From now on, we prove that $|x|u_0 \in L^2(\mathbb{R}^3)$ holds. Fix $R \geq R_0$, where $R_0$ is taken above.

\textit{Step 1:} First, we prove $\frac{d}{dt}J_R > 0$ for any $t \in [0,\infty)$. If not, then there exists $\varepsilon > 0$ and $t_0 \in [0,\infty)$ such that $\frac{d}{dt}J_R[u(t)] < - \varepsilon$ for any $t \geq t_0$ from $\frac{d^2}{dt^2}J_R[u(t)] < 0$.
This contradicts the fact that $J_R[u(t)] \geq 0$ for any $t \in [0,\infty)$.

\textit{Step 2:} We show that $u$ has finite variance. Since $\frac{d}{dt}J_R[u(t)]$ is positive and decreasing, we have $\frac{d}{dt}J_R[u(t)] \rightarrow c$ as $t \rightarrow \infty$ for some $c \geq 0$ and hence,
\begin{align*}
	- \infty
		< c - \frac{d}{dt}J_R[u_0]
		= \int_0^\infty \frac{d^2}{dt^2}J_R[u(s)]ds
		\leq 2\int_0^\infty \delta(s)ds
		\leq 0.
\end{align*}
This inequality implies that there exists a sequence $\{t_n\} \subset [0,\infty)$ with $t_n \rightarrow \infty$ such that $\delta(t_n) \rightarrow 0$ as $n \rightarrow \infty$.
Now, since $u$ is radially symmetric, we can take a sequence $\{\theta_n\} \subset \mathbb{R}$ such that $e^{i\theta_n}u(t_n) \rightarrow Q$ in $H^1$ as $n \rightarrow \infty$ by Lemma~\ref{LemmaMod}.
Therefore, as $J_R[u(t)]$ is increasing, we obtain
\begin{align*}
	J_R[u_0]
		= \int_{\mathbb{R}^3}w_R(x)|u_0(x)|^2dx
		\leq \int_{\mathbb{R}^3}w_R(x)|Q(x)|^2dx
		\leq \int_{\mathbb{R}^3}|x|^2|Q(x)|^2dx
		< \infty.
\end{align*}
Letting $R \rightarrow \infty$, monotone convergence theorem deduces
\begin{align*}
	J_R[u_0]
		\rightarrow \int_{\mathbb{R}^3}|x|^2|u_0(x)|^2dx
		\leq \int_{\mathbb{R}^3}|x|^2|Q(x)|^2dx
		< \infty.
\end{align*}
Therefore, $|x|u_{0}\in L^{2}(\R^{3})$. It follows that $u$ blows up. However, this is contradiction.

This completes the proof of theorem.
\end{proof}

\section{Failure of uniform space-time bounds at threshold}\label{S:Fail}

In this section, we prove Theorem \ref{BoundTheorem}. We follow the proof of Theorem 1.5 in \cite{KillipMurphyVisanZheng}.
\begin{proof}[{Proof of Theorem \ref{BoundTheorem}}]
Consider $\varphi_{n}=(1-\epsilon_{n})Q(x-x_{n})$ with $\epsilon_{n}\to 0$ and $|x_{n}|\to \infty$. 
{Since $\mu\in (1,2)$}, by Lemma 2.7 in \cite{GuoWangYao2018}, we see that
\[E_{V}(\varphi_{n})M(\varphi_{n})\nearrow  E_{0}(Q)M(Q)
\quad \text{and}\quad 
P_{V}(\varphi_{n})\to 0
\]
as $n\to \infty$. Moreover, combining  \eqref{PoQ} and $P_{0}(Q)=0$, it is not hard to show that  $P_{V}(\varphi_{n})>0$ for all $n\in \N$.
Therefore, from Theorem \ref{Th1} we get that the corresponding solution $u_{n}$ to \eqref{NLS} with initial data $\varphi_{n}$
exists globally and scatters.

We want to apply Lemma~\ref{StabilityNLS} over $[-T,T]\times\R^{3}$.
With this in mind, for each $n$, let $\chi_n$ be a smooth function obeying
\[
\chi_n(x)=\begin{cases} 0 & |x_n+ x|<\tfrac14|x_n| \\ 1 & |x_n + x|>\tfrac12 |x_n|,\end{cases}\qtq{with} 
\sup_{x}|\partial^k \chi_n(x)|\lesssim \bigl(\tfrac{\lambda_n}{|x_n|}\bigr)^{|k|}
\]
uniformly in $x$.  Notice that $\chi_n(x)\to 1$ as $n\to\infty$ for each $x\in\R^3$.

Now, fix $T>0$. We put
\[
\tilde{v}_{n}(t,x)=(1-\epsilon_{n})e^{it}[\chi_{n}Q](x-x_{n}).
\]
We need to estimate 
$\||\nabla|^{\frac{1}{2}}e_{n}\|_{N([-T,T])}$,
where
\begin{align}\nonumber
	e_{n}&=(i\partial_{t}-H)\tilde{v}_{n}+|\tilde{v}_{n}|^{2}\tilde{v}_{n}\\ \label{e11}
	&=e^{i t}[(1-\epsilon_{n})^{3}\chi^{3}_{n}(x-x_{n})-(1-\epsilon_{n})\chi_{n}(x-x_{n})]Q^{3}(x-x_{n})\\ \label{e22}
	&+(1-\epsilon_{n})e^{i t}[Q \Delta \chi_{n}+2\nabla \chi_{n}\cdot \nabla Q](x-x_{n})\\  \label{e33}
	&-\tfrac{a}{|x|^{\mu}}(1-\epsilon_{n})e^{it}[\chi_{n}Q](x-x_{n}).
\end{align}

For \eqref{e11}, we apply H\"older's inequality, Sobolev embedding and dominated convergence theorem to estimate 
(recall that $Q\in \Sch(\R^{3})$)
\begin{align*}
&\|\nabla \eqref{e11}\|_{L_{t}^{1}L^{2}_{x}}
\lesssim T[\|\nabla \chi_{n}\|_{L_{x}^{3}} \|Q\|^{3}_{L_{x}^{18}} +\|\nabla Q\|_{L_{x}^{2}} \|Q\|^{2}_{L_{x}^{\infty}}]\lesssim T\\
&\| \eqref{e11}\|_{L_{t}^{1}L^{2}_{x}}\lesssim {{T
 \|Q\|^{2}_{L_{x}^{\infty}}\|[(1-\epsilon_{n})^{3}\chi^{3}_{n}(x-x_{n})-(1-\epsilon_{n})\chi_{n}(x-x_{n})]Q\|_{L^{2}_{x}}}}\to 0.
\end{align*}
Similarly,
\begin{align*}
&\|\nabla \eqref{e22}\|_{L_{t}^{1}L^{2}_{x}}+\|\eqref{e22}\|_{L_{t}^{1}L^{2}_{x}}
\lesssim T[\|\nabla \Delta \chi_{n}\|_{L_{x}^{\infty}}+
\|\Delta \chi_{n}\|_{L_{x}^{\infty}}+\|\nabla  \chi_{n}\|_{L_{x}^{\infty}}
] \|Q\|_{H_{x}^{2}}\\
&\lesssim T[|x_{n}|^{-3}+|x_{n}|^{-2}+|x_{n}|^{-1}]\to 0,\\
	&\|\nabla \eqref{e33}\|_{L_{t}^{1}L^{2}_{x}}+\|\eqref{e33}\|_{L_{t}^{1}L^{2}_{x}}
	\lesssim {{T[\| \tfrac{\chi_{n}}{|\cdot+x_{n}|^{\mu}} \|_{L_{x}^{\infty}} + 
	\|\nabla\( \tfrac{\chi_{n}}{|\cdot+x_{n}|^{\mu}} \)\|_{L_{x}^{\infty}}]\|Q\|_{H_{x}^{1}}}}\\
	&\lesssim T[|x_{n}|^{-\mu}+|x_{n}|^{-(\mu+1)}]\to 0,
\end{align*}
as $n\to \infty$.
Therefore, for any $T > 0$ fixed, by interpolation we get
\[
\||\nabla|^{\frac{1}{2}}e_{n}\|_{N([-T,T])}\to 0 \qtq{as $n\to \infty$.}
\]
But then, since
\begin{align*}
	&\|\tilde{v}_{n}(0)-{{{\varphi}_{n}}}\|_{\dot{H}^{\frac{1}{2}}}=
	\|(1-\epsilon_{n})(\chi_{n}-1)Q\|_{\dot{H}^{\frac{1}{2}}}\to 0,\\
	&\| \tilde{v}_{n}  \|_{L^{5}_{t,x}([-T, T]\times \R^{3})}\gtrsim_{Q}T \qtq{for any $T>0$,}
\end{align*}
Lemma \ref{StabilityNLS} implies that 
\[
\| {u}_{n}  \|_{L^{5}_{t,x}([-T, T]\times \R^{3})}\gtrsim_{Q}T
\]
which finished the proof because $T>0$ is arbitrary.
\end{proof}


\begin{thebibliography}{10}

\bibitem{AkahoriNawa2013}
{\sc T.~Akahori and H.~Nawa}, {\em Blowup and scattering problems for the
  nonlinear {S}chr\"odinger equations}, Kyoto J. Math., 53 (2013),
  pp.~629--672.

\bibitem{Ardila2022}
{\sc A.~H. Ardila}, {\em On blow-up of the threshold solutions for the focusing
  {NLS} with a repulsive potential}, Preprint,  (2021).

\bibitem{ArdilaInui2022}
{\sc A.~H. Ardila and T.~Inui}, {\em Threshold scattering for the focusing
  {NLS} with a repulsive dirac delta potential}, J. Differ. Equ., 313 (2022),
  pp.~54--84.



\bibitem{Dinh2021N}
{\sc V.~D. Dinh}, {\em Non-radial scattering theory for nonlinear
  {S}chr\"odinger equations with potential}, Nonlinear Differ. Equ. Appl., 28
  (2021).

\bibitem{DodsonMurphy2018}
{\sc B.~Dodson and J.~Murphy}, {\em A new proof of scattering below the ground
  state for the non-radial focusing {NLS}}, Math. Res. Lett., 25 (2018).

\bibitem{DuyLanRou2022}
{\sc T.~Duyckaerts, O.~Landoulsi, and S.~Roudenko}, {\em Threshold solutions in
  the focusing {3D} cubic {NLS} equation outside a strictly convex obstacle},
  J. Funct. Anal., 282 (2022), p.~109326.

\bibitem{DuyckaertsRou2010}
{\sc T.~Duyckaerts and S.~Roudenko}, {\em Threshold solutions for the focusing
  {3D} cubic {S}chr\"odinger equation}, Rev. Mat. Iberoam., 26 (2010),
  pp.~1--56.

\bibitem{GuoWangYao2018}
{{\sc Q.~Guo, H.~Wang and X.~Yao}, {\em Dynamics of the focusing 3{D} cubic {NLS}
  with slowly decaying potential}}, J. Math. Anal. Appl., 56 (2022), p.~125653.

\bibitem{HamanoIkeda20}
{\sc M.~Hamano and M.~Ikeda}, {\em Equivalence of conditions on initial data
  below the ground state to {NLS} with a repulsive inverse power potential},
  Preprint arXiv:2004.08788.

\bibitem{2HamanoIkeda2022}
\leavevmode\vrule height 2pt depth -1.6pt width 23pt, {\em Scattering solutions
  to nonlinear Schr\"odinger equation with a long range potential}, Preprint
  arXiv:2104.13577.

\bibitem{TInu2022}
{\sc T.~Inui}, {\em Remark on blow-up of the threshold solutions to the
  nonlinear Schr\"odinger equation with the repulsive dirac delta potential},
  Preprint,  (2021).

\bibitem{KillipMurphyVisanZheng}
{\sc R.~Killip,  J.~Murphy, M.~Visan, and J.~Zheng}, {\em The focusing cubic NLS with
  inverse square potential in three space dimensions}, Diff. Inte. Equ., 30
  (2017), pp.~161--206.

\bibitem{SLC}
{\sc S.~{Le~Coz}}, {\em Standing waves in nonlinear {S}chr{\"o}dinger
  equations}, In: Analytical and Numerical Aspects of Partial Differential
  Equations, Walter de Gruyter, Berlin,  (2009), pp.~151--192.

\bibitem{LU20183174}
{\sc J.~Lu, C.~Miao, and J.~Murphy}, {\em Scattering in ${H}^{1}$ for the
  intercritical {NLS} with an inverse-square potential}, J. Differential
  Equations, 264 (2018), pp.~3174--3211.

\bibitem{MiaMurphyZheng2021}
{\sc C.~Miao, J.~Murphy, and J.~Zheng}, {\em Threshold scattering for the
  focusing {NLS} with a repulsive potential}, Preprint arXiv:2102.07163, To
  appear in Indiana Univ. J. Math.,  (2021).

\bibitem{MiaoZhanZheng2018}
{\sc C.~Miao, J.~Zhang, and J.~Zheng}, {\em Nonlinear {S}chr\"odinger equation
  with {C}oulomb potential}, Preprint, arXiv:1809.06685.,  (2018).


\bibitem{ZhangZheng2014}
{\sc J.~Zhang and J.~Zheng}, {\em Scattering theory for nonlinear Schr\"odinger equations with inverse-square potential}, J. Funct. Anal., 267
  (2014), pp.~2907--2932.
	
\bibitem{Xu} G.~Xu, {\em Dynamics of some coupled nonlinear Schr\"odinger systems in $\mathbb{R}^3$}, Math. Methods Appl. Sci. {37} (2014), no. 17, 2746--2771.

\end{thebibliography}

\end{document}